\documentclass{article}

\usepackage[english]{babel}
\usepackage{graphicx} 
\usepackage{amssymb}
\usepackage{amsmath}
\usepackage{amsthm}
\usepackage{fullpage}
\usepackage{mathrsfs}
\usepackage{subfigure}
\usepackage{verbatim}
\usepackage{enumitem}
\usepackage{bbm}
\usepackage{tikz}
\usetikzlibrary{through,calc,positioning,decorations.pathreplacing, arrows.meta}
\usepackage{authblk}
\usepackage{indentfirst}

\counterwithin{figure}{section}

\usepackage[letterpaper,top=2cm,bottom=2cm,left=3cm,right=3cm,marginparwidth=1.75cm]{geometry}

\usepackage[colorlinks=true, allcolors=blue]{hyperref}

\renewcommand{\P}{\mathbb{P}}
     
     \newcommand{\E}[0]{\mathbb{E}}
     \newcommand{\floor}[1]{\left\lfloor#1\right \rfloor}

\newcommand{\N}{\mathbb{N}}

 \newcommand{\df}[1]{\textbf{#1}}

\def\Z{\mathbb{Z}}

\newcommand{\bZ}{{\mathbb Z}}

\newcommand{\cU}{\mathcal U}

\newcommand{\cO}{\mathcal O}

\numberwithin{equation}{section}
     
\newtheorem{theorem}{Theorem}[section]

\newtheorem{lemma}[theorem]{Lemma}

\newtheorem{prop}[theorem]{Proposition}

\newtheorem{remark}[theorem]{Remark}

\definecolor{ashgrey}{rgb}{0.7, 0.75, 0.71}
\definecolor{faint}{rgb}{0.9, 0.9, 0.9}
\definecolor{lgrey}{rgb}{0.85, 0.85, 0.85}
\definecolor{dgrey}{rgb}{0.5, 0.5, 0.5}
\definecolor{mgrey}{rgb}{0.78, 0.78, 0.78}
\definecolor{lred}{rgb}{1, 0.8, 0.8}
\definecolor{dred}{rgb}{0.7, 0, 0}

\title{Two-stage Bootstrap Percolation}

\author[1]{Zihao Fang}
\author[2]{Janko Gravner}
\author[3]{David Sivakoff}
\affil[1]{Department of Mathematics, The Ohio State University, {\tt fang.735@osu.edu}}
\affil[2]{Department of Mathematics, University of California, Davis, 
{\tt gravner@math.ucdavis.edu}}
\affil[3]{Departments of Statistics and Mathematics, The Ohio State University, {\tt dsivakoff@stat.osu.edu}}

\begin{document}
\maketitle

\begin{abstract}
We introduce and study two variants of two-stage growth dynamics in $\bZ^2$ with state space $\{0,1,2\}^{\bZ^2}$. 
In each variant, vertices in state $0$ can be changed irreversibly to state $1$, and vertices in state $1$ can be changed permanently to state $2$. In the standard variant, a vertex flips from state $i$ to $i+1$ if it has at least two nearest-neighbors in state $i+1$. In the modified variant, a $0$ changes to a $1$ if it has both a north or south neighbor and an east or west neighbor in state $1$, and a $1$ changes to a $2$ if it has at least two nearest-neighbors in state $2$. We assume that the initial configuration  is given by a product measure with small probabilities $p$ and $q$ of $1$s and $2$s. For both variants, as $p$ and $q$ tend to $0$, if $q$ is large compared to $p^{2+o(1)}$, then the final density of $0$s tends to $1$. When $q$ is small compared to $p^{2+o(1)}$, for standard variant the final density of $2$s tends to $1$, while for the modified variant the final density of $1$s tends to $1$. In fact, for the modified variant, the final density of $2$s approaches $0$ regardless of the relative size of $q$ versus $p$. These results remain unchanged if, in either variant, a $1$ changes to a $2$ only if it has both a north or south neighbor and an east or west neighbor in state $2$.
An essential feature of these dynamics is that they 
are not monotone in the initial configuration.
\end{abstract}

\section{Introduction}

\subsection{Background and setup}
Arguably, 
the most basic and influential mathematical model for nucleation and growth is {\it bootstrap percolation\/}, a process that iteratively enlarges a set of
occupied vertices 
of a graph by adjoining vertices with at least a threshold number 
of already occupied neighbors. Bootstrap 
percolation and its generalizations have been studied from
a variety of aspects, yielding many surprising and deep results, of which we just list a few highlights \cite{AizenmanLebowitz, Holroyd2003,  Balogh2012, Bollobs2022, BBMS2022, BBMS22} and refer to the survey paper \cite{Morris2017} for a more comprehensive discussion. The process on
two-dimensional lattice 
with threshold $2$ remains the most well-known instance and a 3-state version is the subject of
the present paper.

We call our rule \textit{two-stage bootstrap percolation}. 
At each integer time $t\ge 0$, the evolving configuration is 
denoted by $\xi_t\in \{0,1,2\}^{\bZ^2}$.
In the initial configuration $\xi_0$, each vertex
$x\in\bZ^2$ is in one of three states $0$, $1$, $2$ 
according to the uniform 
product measure with probabilities
\begin{equation}
\label{eqn:initialconfiguration}
    \P(\xi_0(x) = 1) = p, \qquad \P(\xi_0(x) = 2) = q, \qquad  \text{and} \qquad \P(\xi_0(x) = 0) = 1 - p - q,
\end{equation}
where $p + q < 1$. We will assume $p,q>0$ unless explicitly stated otherwise.
We often make a short-hand reference to sites in state $0$, $1$, or $2$ as $0$s, $1$s, or $2$s. 
For $x\in\bZ^2$, we denote by $N_i(x,t)\in[0,4]$, 
$N_i^h(x,t)$, and $N_i^v(x,t)\in[0,2]$ the 
number of vertices in state $i$ at time $t$ among the four nearest 
neighbors of $x$, and among the two nearest vertical and 
horizontal neighbors of $x$. Then we let
$$
N_i'(x,t)=\mathbbm 1_{\{N_i^h(x,t)\ge 1\}}+\mathbbm 1_{\{N_i^v(x,t)\ge 1\}}\in [0,2] 
$$
be the 
number of coordinate directions that have a nearest neighbor of $x$ in state $i$ at time $t$. In the {\it standard variant\/}, 
the deterministic update rule is then given as follows:
\begin{itemize}
\item[(D0)] if $\xi_t(x)=0$, then $\xi_{t+1}(x)=1$ if $N_1(x,t)\ge 2$ and $\xi_{t+1}(x)=0$ otherwise; 
\item[(D1)] if $\xi_t(x)=1$, then $\xi_{t+1}(x)=2$ if $N_2(x,t)\ge 2$ and $\xi_{t+1}(x)=1$ otherwise;
\item[(D2)] if $\xi_t(x)=2$, then $\xi_{t+1}(x)=2$.
\end{itemize}
Thus, $1$s perform threshold-2 bootstrap percolation on $0$s, $2$s perform threshold-2 bootstrap percolation on $1$s, and state $2$ is terminal. 
In the {\it modified variant\/}, $N_1(x,t)$ in (D0) is replaced by 
$N_1'(x,t)$. We emphasize that all our theorems and proofs remain valid 
if we also replace $N_2(x,t)$ by 
$N_2'(x,t)$ in (D1), but for concreteness our 
modified variant has (D1) unchanged.

The classic result \cite{Holroyd2003} 
for the two-color dynamics --- the one with $q=0$ --- 
concludes
that the standard and modified variant both have exponentially rare nucleation in $1/p$. More precisely, if $T$ is the first time at which the 
origin becomes occupied, then $p\log T$ converges to a constant, in probability, as $p\to 0$; the limiting constant is $\pi^2/18$ for the 
standard and $\pi^2/6$ for the modified variant. More recent results determined the asymptotic behavior with great precision \cite{Hartarsky2019},
\cite{Hartarsky2024},
and also established the difference between two two variants in the scaling of second-order term \cite{Hartarsky2023}. 

Closer to our model is the {\it polluted 
bootstrap percolation\/} \cite{GravnerMcDonald}, where $2$s 
only act as obstacles without otherwise influencing $1$s; in our context, 
we can interpret 
this as the transition from $1$ to $2$ requiring impossible-to-satisfy threshold $5$, that is, with the condition $N_2(x,t)\ge 5$ in (D1). The fundamental 
question in this case is 
how the 
the relative size of $q$ vs.~$p$ 
influences the density of $1$s in the final configuration, 
as $(p,q)\to(0, 0)$. For the standard variant, the origin's final state is with high probability 
$1$ when 
$q\le cp^2$, while the origin is highly likely to remain forever $0$ when $q\ge Cp^2$ \cite{GravnerMcDonald}.
For 
the modified variant, the transition is between $q\le cp^2/(\log p^{-1})^{1+o(1)}$ and $q\ge Cp^2/(\log p^{-1})$ \cite{GHLS2025}. Thus, 
in the case of polluted bootstrap percolation, there is a disparity in the critical scaling 
between the two variants, albeit only logarithmic.

The difference in the update rule between 
the two variants may be considered minor and in the above two cases, and the discrepancy in the behavior is also rather subtle. 
By contrast, in our setting, 
the discrepancy is qualitative, as the nature of the 
phase transition is altered: for small $q$ the 
final state is dominated by $2$s in the standard variant and by $1$s in the 
modified one (see Theorem~\ref{thm:phase-transition}). Consequently, 
universality results for monotone rules \cite{Bollobs2022, BBMS22, Ghosh} do not have clear counterparts for $3$-color two-stage models. 

Similar models have been studied from variety of perspectives and 
we conclude this subsection with a brief review. Works on higher dimensional  
polluted models include \cite{GravnerHolroyd, GravnerHolroydSivakoff, Ding}. The paper  \cite{BresarHedzetHenning} considers extremal 
problems related to polluted dynamics. 
In an interacting particle system setup, \cite{Kordzakhia} 
studies the \textit{escape model} (later known as the \textit{chase-escape model}), where the ``predator" particles grow on the ``prey" particles, while the prey colonize the empty vertices. In common with our two-stage 
bootstrap percolation, the possible updates in
chase-escape model form an oriented linear graph on $3$ states. By contrast, in the cyclic update rules such a graph is an oriented cycle.
In this vein, \cite{Fisch} considered the \textit{cyclic cellular automaton} with state space $\{1,2,\dots, N\}^{\bZ^d}$, where a state-$i$ vertex changes its state to $(i + 1 )\ (\operatorname{mod} N)$ if by contact with a state-$(i+1) \ (\operatorname{mod} N)$ neighbor in each discrete-time step. Its continuous-time analogue, the \textit{cyclic particle system}, was analyzed in \cite{BramsonGriffeath2}. Two-dimensional threshold versions akin to ours were introduced in \cite{Fisch1991}.

Transitions among multiple types appear in many other models.
For example, 
\cite{CourtBlytheAllen} introduced the {\it stacked contact process\/}, in which transitions between any pair of states 0 (empty), 1 (healthy), 2 (infected) are allowed, as infection spreads through neighbors and offspring. We refer to 
the survey book \cite{Lanchier} for a review of recent development on such models.

\subsection{Main results}


Recall that our standard variant evolves according to (D0)--(D2), while the modified variant has $N$ replaced by $N'$ in (D0). We remark again that $N$ may be replaced by $N'$ in (D1) in either variant, without any change in the statements and proofs. One of the reasons for interest in these
dynamics is its fundamental nonmonotonicity in the initial configuration: 
adding more $2$s initially, 
for example, does not necessarily result in more $2$s later on. This is reflected in Theorem~\ref{thm:phase-transition} and presents the main 
technical challenge in our arguments.

\begin{theorem}\label{thm:0slose}
    For either variant of two-stage bootstrap percolation, if  $q \ll p^2/(\log(1/p))^2$, then
    \[
    \P(\text{the origin is eventually not in state 0}\, ) \to 1
    \]
    as $p \to 0$.
\end{theorem}

\begin{theorem}\label{thm:2swininstandard}
    For the standard variant of two-stage bootstrap percolation, if $q \ll p^2/(\log(1/p))^2$, then
    \[
    \P(\text{the origin is eventually in state 2}\, ) \to 1
    \]
    as $p \to 0$.
\end{theorem}

\begin{theorem}\label{thm:2sloseinmodified}
    For the modified variant of two-stage bootstrap percolation,
    \[
    \P(\text{the origin is never in state 2}\, ) \to 1
    \]
    as $(p, q) \to (0,0)$.
\end{theorem}

The following theorem easily follows from the preceding theorems and the results on polluted bootstrap percolation
in \cite{GravnerMcDonald, GravnerHolroydLeeSivakoff}.

\begin{theorem}\label{thm:phase-transition}
    For the standard variant, if $q \ll p^2/(\log(1/p))^2$, then
    \[
    \P(\text{the origin is eventually in state 2}\, ) \to 1,
    \]
    and if $q\gg p^2$, then
    \[
    \P(\text{the origin remains forever in state 0}\, ) \to 1.
    \]
    For the modified variant, if $q \ll p^2/(\log(1/p))^2$, then
    \[
    \P(\text{the origin is eventually in state 1}\, ) \to 1,
    \]
    and if $q\gg p^2/\log(1/p)$, then
    \[
    \P(\text{the origin remains forever in state 0}\, ) \to 1.
    \]

\end{theorem}

The phase transition established in Theorem~\ref{thm:phase-transition} for the standard variant is illustrated 
in Figure~\ref{fig:sim}.

\begin{figure}[h!]
         \centering
         \includegraphics[width=0.2\linewidth]{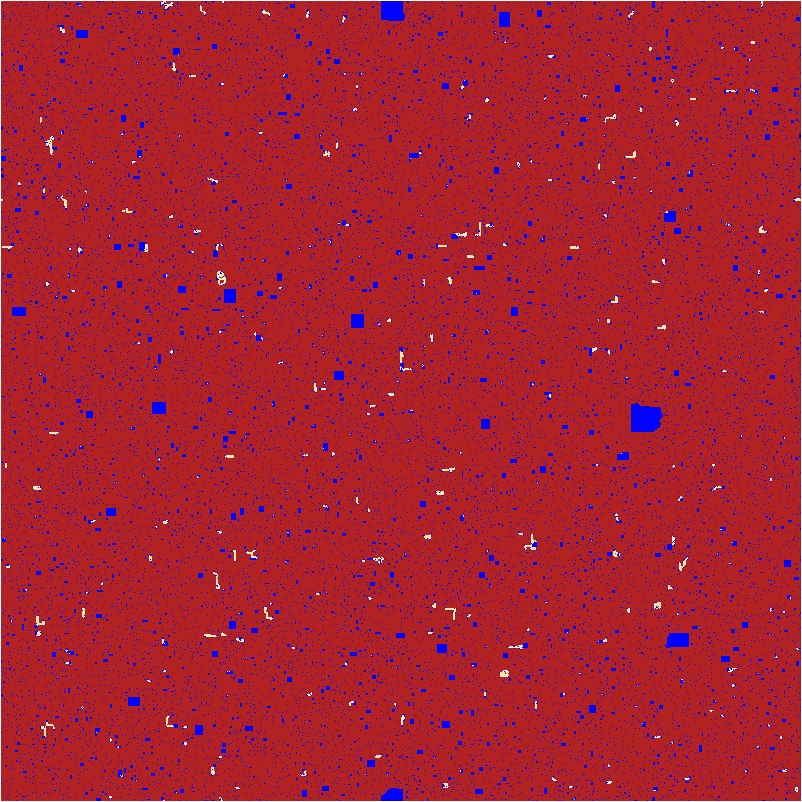}
         \includegraphics[width=0.2\linewidth]{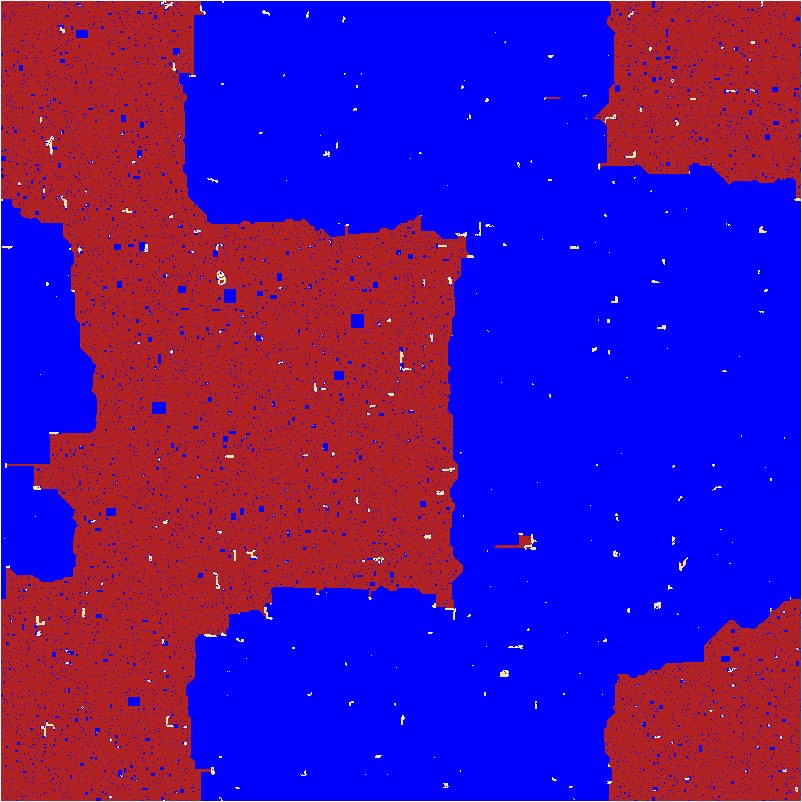}
         \includegraphics[width=0.2\linewidth]{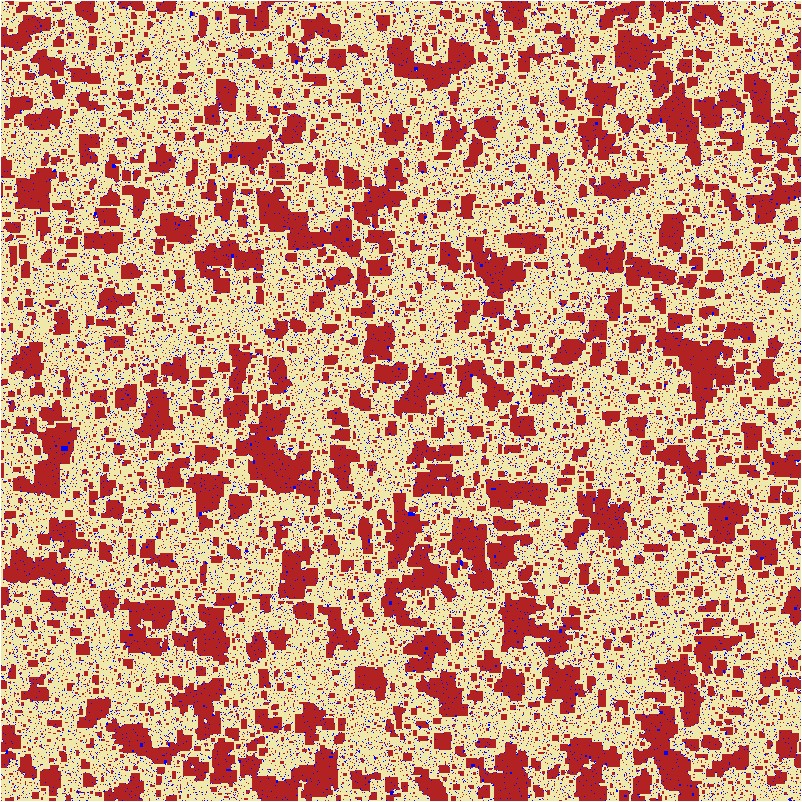}
          \includegraphics[width=0.2\linewidth]{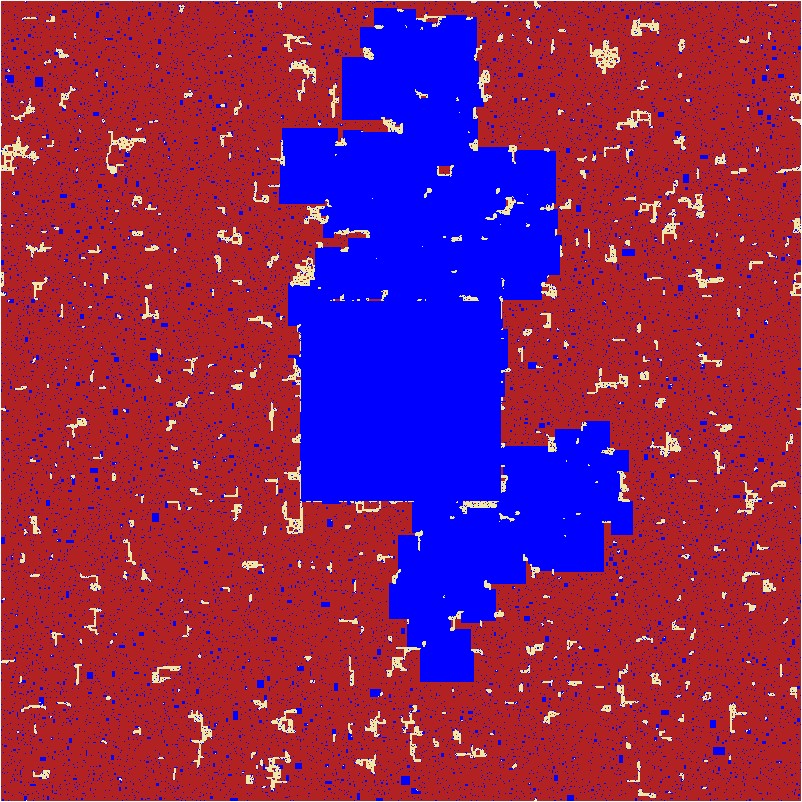}
         \caption{Simulations of the standard variant of two-stage bootstrap percolation with $q=0.04$ on $800\times 800$ discrete torus. The left two frames depict the same dynamics with $p=0.2$
         at two different times: at time $80$, most sites are in state $1$ (red); while at time $1000$, the $2$s (blue) have grown from nucleation centers to overtake
         a significant portion of the space and are on the way to controlling most of it. (It is challenging to design simulations in which 
         nucleation of $1$s and $2$s are both rare, due to 
         vastly different scales.) The middle frame is the final state with $p=0.08$, whereby
         $0$s (yellow) dominate (albeit not overwhelmingly). The final picture has intermediate $p=0.14$ and is at the final state. From the product measure, most sites in the square become
         permanently $1$ at this density, so we added a central $200\times 200$ 
         square of $2$s to test whether $2$s can nucleate; while the square 
         grows significantly, it is eventually stopped by remaining $0$s.}
         \label{fig:sim}
     \end{figure}

       \subsection{Organization and main ideas in the proofs}
       The 
       proof of Theorem~\ref{thm:0slose}  is given in Section~\ref{sec:eliminate-0s}, where we demonstrate that $1$s 
       are likely able to spread through boxes of diameter on the 
       order $(1/p)\log(1/p)$, from a box that is internally 
       filled by $1$s to the origin, and accomplish this 
       without interference by $2$s. This is an adaptation of a familiar argument  in polluted bootstrap percolation \cite{Gravner1997, GravnerHolroydSivakoff, GravnerHolroydLeeSivakoff}. 
       
       Our major effort 
       is devoted to the proof of Theorem~\ref{thm:2swininstandard} in Section~\ref{sec:occupation-2-standard}. The first two frames in Figure~\ref{fig:sim} suggest
       separation of scales: $1$s must dominate the 
       space long before $2$s are able to achieve much. 
       Our rigorous argument establishes this scenario and thus has to balance the scarcity of $2$s required for $1$s to spread and the presence of $2$s for them 
       to grow later on. Accordingly, we introduce boxes 
       of diameter on the order of $(1/q)\log(1/q)$ which can be covered 
       by $2$s provided: (1) {\it all\/} $0$s within them are eliminated; and (2) they border a box of the same size which has already been taken over by $2$s. The property (1) can be achieved by 
       constructing a circuit of vertices that eventually become $1$ 
       and then ensuring that within that circuit $0$s cannot persist; this elimination of 0s only occurs for the standard variant and is responsible for the difference between the two variants. The property 
       (2) is again a standard polluted bootstrap percolation construction. 
       Observe that the remaining $0$s present obstacles for the spread of $2$s; those obstacles will be there (see again the second frame of Figure~\ref{fig:sim}), but they are not frequent enough to significantly impede
       the growth of $2$s. 

       The relative scarcity of $2$s is the key to the argument sketched 
       so far, but now we have to ensure that the spread of $2$s can be 
       ``ignited.'' For this, we make no attempt at efficiency. Instead, 
       after we construct an infinite connected cluster of boxes that 
       satisfy (1) and (2) above, we simply show 
       that it is {\it possible\/} that a configuration of bounded size
       (albeit exponential in $1/q$)
       borders the cluster and is eventually filled 
       by $2$s. Then, ergodicity 
       implies that that such a configuration must appear, causing the 
       entire cluster, which contains the origin with high probability, 
       to be eventually covered by $2$s.

       By contrast, our proof of Theorem~\ref{thm:2sloseinmodified} in Section~\ref{sec:preventspreadof2s} establishes that, 
       in the modified variant, the $0$s which 
       never turn to $1$ {\it are\/} frequent enough to stop the spread 
       of $2$s. To see why, note that the $0$ in a ``$202$" configuration never changes its state. Such a configuration occurs with density about $q^2$, 
       which is already on the same order as the critical density for 
       the polluted model \cite{Gravner1997}. (Here, $2$s play the role of occupied sites and permanent 
       $0$s are polluted sites.) By considering any finite interval of 
       $0$s between two $2$s, we can increase the density of obstacles
       to any constant times $q^2$, making them supercritical.  We then 
       carry out the rescaling scheme from \cite{GS2020}, which is based on a ``protective shell'' construction obtained through duality methods used to construct Lipschitz random surfaces \cite{LipPerc, GrimmettHolroyd10, GrimmettHolroyd12}. Some 
       additional care is necessary due to non-monotonicity: 
       $2$s are needed 
       to ensure the presence of permanent $0$s, but are
       otherwise required to be absent near the boundary of the protective shell. 

       We conclude the paper with a discussion on related models and 
       a selection of open problems.

\section{Preliminaries}

We start by clarifying some notions that will be used substantially later.
For a set $S\subseteq \bZ^2$, we define the \textbf{internal dynamics} on $S$ as follows. 
We assume that the initial configuration on $S$ is inherited from the initial configuration on $\bZ^2$, and restrict the neighborhood of every $x\in S$ to the nearest neighbors of $x$ in $S$. 
Then we run the (modified or standard) dynamics from this initial configuration and with the restricted neighborhoods.

For either dynamics (which will be clear from the context), we say that a set $S\subseteq \bZ^2$ is \df{internally spanned by $1$s} if in the internal dynamics on $S$, the final configuration results in $S$ being fully occupied by $1$s. 
Note that a set $S$ may be internally spanned by $1$s, but in the dynamics on $\bZ^2$, $S$ may contain $0$s and $2$s in the final configuration. Also note that if $S$ is internally spanned by $1$s for the modified variant, then it is also internally spanned by $1$s for the standard variant.

The following lemma, which will be used in later sections, is a variant of a result due to Aizenman and Lebowitz \cite{AizenmanLebowitz}.
For a rectangle $R \subset \bZ^2$, denote by $\operatorname{long}(R)$ the length of the longest side of $R$.

\begin{lemma}[Aizenman-Lebowitz]
\label{lem:Aizenman-Lebowitz}
    Suppose that the initial configuration on a finite square box $B \subset \bZ^2$ contains only $1$s and $2$s, and run the bootstrap percolation dynamics of $2$s on $1$s internal to $B$.
    Consider an arbitrary rectangle $S \subset B$ that is completely filled by $2$s in the final configuration of the internal dynamics on $B$.
    Then for every $j \in \N$ with $j \leq \operatorname{long}(S)$, 
    there is a rectangle $R \subset B$ such that the internal dynamics on $R$ completely fills $R$ with $2$s in its final configuration, and $\operatorname{long}(R) \in [j/2, j]$.
\end{lemma}

Note that Lemma \ref{lem:Aizenman-Lebowitz} works for both the standard and modified percolation rule when $1$s are flipped to $2$s.

We will need a version of the second moment method given in the following lemma.

\begin{lemma}\label{lem:2nd-moment} Fix an $\epsilon>0$.
Assume that $A_i$, $i=1,\ldots, I$, are events such that $\P(A_i\cap A_j)\le \P(A_i)\P(A_j)(1+\epsilon)$ for all $i\ne j$, and $\sum_i\P(A_i)\ge 1/\epsilon$. Then 
$\P(\cup_iA_i)\ge 1-2\epsilon$. 
\end{lemma}

\begin{proof} Let  
$X := \sum_i \mathbbm 1_{A_i}$.  Then 
$$
\E(X^2)=\E X+\sum_{i\ne j} \P(A_i\cap A_j)\le \E X+(1+\epsilon)\sum_{i,j} \P(A_i)\P(A_j)=\E X+(1+\epsilon)(\E X)^2, 
$$
and so, by the second moment method, 
\begin{equation*}
\begin{aligned}
\P(X>0)\ge \frac{(\E X)^2}{\E(X^2)} 
\ge \frac{\E X}{1+(1+\epsilon)\E X}
\ge 
\frac{1}{1+2\epsilon}
\ge 
1-2\epsilon.
\end{aligned}
\end{equation*}
\end{proof}

Throughout the paper, we will use the following common asymptotic notations:

\begin{itemize}
    \item 
    $f = \cO(g)$, or equivalently, $f \lesssim g$, $g \gtrsim f$, if $\limsup f/g < \infty$; and
    \item 
    $f = o(g)$, or equivalently, $f \ll g$, $g \gg f$, if $\lim f/g = 0$.
\end{itemize}
All limits are taken as the variable ($p$ or $q$) in the context approaches $0$ from the positive side, unless otherwise specified.

\section{Eliminating $0$s when $q$ is small}\label{sec:eliminate-0s}

We will prove Theorem \ref{thm:0slose} in this section.
Namely, we will show that under either variant, 
\[
\P(\text{the origin is eventually not in state 0}) \to 1
\]
as $p \to 0$ and $q \ll p^2/(\log(1/p))^2$. 
For this, it suffices to show that if the origin starts in state $0$, then it will flip to state $1$ at some time.

A \textbf{$p$-box} is a square with side length $2\floor{\frac{1}{p}\log(1/p)}$. We tile the plane with $p$-boxes in checkerboard fashion, and we call two $p$-boxes adjacent if they contain vertices that are nearest neighbors.
We say that a $p$-box is \textbf{$1$-crossable} if initially every row and every column of this $p$-box contains a $1$, and there is no $2$ not only in this $p$-box itself, but also in the eight $p$-boxes surrounding this $p$-box.
A $p$-box is called a \textbf{$1$-nucleus} if it is internally spanned by $1$s in the modified dynamics.

\begin{lemma}\label{lem:p-boxcrossablewhp}
    A $p$-box is $1$-crossable with high probability provided that $q \ll \frac{p^2}{(\log(1/p))^2}$.
\end{lemma}
\begin{proof}
    By the union bound, the probability that a $p$-box is not $1$-crossable is at most
    \[
    \left(3\cdot \frac{2}{p}\log\left(\frac{1}{p}\right)\right)^2 q + 2\left(\frac{2}{p}\log\left(\frac{1}{p}\right)\right)(1 - p)^{\frac{2}{p}\log(1/p) - 2} \to 0
    \]
    as $p \to 0$ with $q \ll \frac{p^2}{(\log(1/p))^2}$.
\end{proof}
Let $B_0 := [-M, M]^2$ be the square box centered at the origin, where $M := \floor{e^{C/p}}$ with the constant $C$ to be specified later. 

\begin{lemma}\label{lem:highway}
    As $p \to 0$ with $q \ll \frac{p^2}{(\log(1/p))^2}$, the following holds with high probability. 
    There is a path of adjacent $1$-crossable $p$-boxes that travels from the $p$-box containing the origin to a $p$-box that intersects the boundary of $B_0$.
\end{lemma}

\begin{proof}
    First, by Lemma \ref{lem:p-boxcrossablewhp}, the $p$-box containing the origin is itself $1$-crossable with high probability.
    If such a path of $1$-crossable $p$-boxes does not exist, then within distance $1/p^2$ of $B_0$ there must exist a circuit of non-$1$-crossable $p$-boxes surrounding the origin.
    Such a circuit may take diagonal steps.
    Note that starting from any location, the number of paths (allowing diagonal steps) of length $j$ is at most $8 \cdot 7^{j-1}$.
    Let $\rho = \rho(p)$ be the probability that a $p$-box is not $1$-crossable, and by Lemma \ref{lem:p-boxcrossablewhp}, $\rho \to 0$ as $p \to 0$.
    It follows that the expected number of such circuits surrounding the origin is at most
    \[
    \sum_{j \geq 4} j \cdot 8 \cdot 7^{j-1} \cdot \rho^{j/10},
    \]
    which converges and becomes arbitrarily small as $p \to 0$. In the expression above, the exponent $j/10$ comes from the fact that $1$-crossability of a $p$-box is dependent on the configuration in itself and in the eight surrounding boxes.
\end{proof}



We call such a path of $1$-crossable $p$-boxes described in Lemma \ref{lem:highway} a \textbf{highway}.
Given a highway, we will find a $1$-nucleus on it so that the $1$-nucleus will transmit $1$s to the origin through the $1$-crossable $p$-boxes on the highway.
By the definition of $1$-crossable $p$-boxes, the transmission of $1$s proceeds for both variants.

\begin{lemma}\label{lem:1-nucleus}
    The probability that a $p$-box is a $1$-nucleus is at least $e^{-\pi^2/(3p)}$ for all $p \in (0,1)$.
\end{lemma}
\begin{proof}
    For convenience, we provide the standard proof of this lower bound.
    Consider the $j$ by $j$ sub-square located at the lower-left corner of the $p$-box (the sub-square shares the left and bottom boundaries with the $p$-box).
    A sufficient condition for a $p$-box to be a $1$-nucleus is that for each $j = 1,2,\dots, 2\floor{\frac{1}{p}\log(1/p)}$, there is a 1 in its top external boundary and a 1 in its right external boundary in the initial configuration.
    The probability for this condition is at least 
    \[
    \prod_{j = 1}^\infty (1 - (1 - p)^j)^2 
    \geq 
    \exp\left(\frac{2}{p}\sum_{j = 1}^\infty p \log(1 - e^{-pj})\right)
    \geq \exp\left(\frac{2}{p} \int_0^\infty \log(1 - e^{-x})dx\right).
    \]
    By the Taylor expansion $\log(1 - e^{-x}) = \sum_{j = 1}^\infty e^{-jx}/j$, the improper integral evaluates to $-\pi^2/6$.
    Plugging this into the expression above gives the desired bound.
\end{proof}

\begin{lemma}\label{lem:find1-nucleus}
    Suppose $C > \pi^2/3$. With high probability, there exists a path of length at most $p^4 M$ of adjacent $1$-crossable $p$-box from the $p$-box containing the origin to a $1$-nucleus.
\end{lemma}
\begin{proof}
    By Lemma~\ref{lem:highway} there exists a highway with high probability. Conditional on this highway, consider its first $p^4 M$ $p$-boxes starting with the box containing the origin.
    The events that a $p$-box is a $1$-nucleus and that a $p$-box is $1$-crossable are positively correlated.
    Thus, by Lemma \ref{lem:1-nucleus}, the conditional (on the highway) probability that each of these $p^4M$ $p$-boxes fails to be a $1$-nucleus is at most
    \[
    (1 - e^{-\pi^2/(3p)})^{p^4 M/10},
    \]
    which vanishes as $p \to 0$ if $C > \pi^2/3$ (recall that $M = \floor{e^{C/p}}$).
\end{proof}

Meanwhile, we need to ensure that the $2$s from outside do not interfere with the transmission of $1$s to the origin.

\begin{lemma}\label{lem:nointerventionof2s}
    When $q \leq p^2$, the bootstrap percolation of $2$s internal to the region within  $\ell^\infty$-distance $M$ of $B_0$ produces no connected component of $2$s with diameter larger than $\frac{1}{p}\log(1/p)$ with high probability.
\end{lemma}
\begin{proof}
    Consider the larger box with dimensions $5M$ by $5M$ containing $B_0$ at its center.
    In the initial configuration, flip all the state 0 vertices in the large box to state 1, and consider the bootstrap percolation dynamics of $2$s on $1$s internal to this larger box.
    Then, all the maximal connected components of state-$2$ vertices in the final configuration of this internal dynamics are rectangles. 
    For $j \in \N$, let $E_j$ be the event that the final configuration contains a rectangle of state-$2$ vertices with the longest side of length at least $j$.
    When $E_j$ occurs, by Lemma \ref{lem:Aizenman-Lebowitz}, the $5M$ by $5M$ box contains a rectangle $R$ with $\operatorname{long}(R) \in [j/2, j]$ and whose internal dynamics fills it entirely with $2$s.
    Then, every pair of neighboring lines that intersect $R$ and are perpendicular to the longest side of $R$ must contain a state-$2$ vertex within $R$ in the initial configuration.
    Meanwhile, two pairs of neighboring parallel lines satisfy this condition independently if they do not overlap.
    The number of such rectangles $R$ within the $5M$ by $5M$ box is at most $(5M)^2 j^2$.
    Using the fact that $1 - q \geq e^{-2q}$ for all small enough $q$, we have 
    \[
    \P(E_j) \leq 25 M^2 j^2 (1 - (1 - q)^{2j})^{j/4 - 1} 
    \leq
    25 M^2 j^2(1 - e^{-4jq})^{j/4 - 1}
    \leq
    25 M^2 j^2\exp(-(j/4 - 1)e^{-4jq})
    .
    \]
    Take $j = \floor{(1/p)\log(1/p)}$ and note that $M \leq e^{C/p}$.
    The expression above vanishes as $p \to 0$ if $q \leq p^2$.
\end{proof}

\begin{proof}[Proof of Theorem \ref{thm:0slose}]
By Lemma \ref{lem:nointerventionof2s}, the growth of $2$s will not enter a $1$-crossable $p$-box in $B_0$ before time $M$. In particular, the guaranteed to exist by Lemma~\ref{lem:highway} is uninhibited by $2$s through time $M$.
Then, by Lemma \ref{lem:find1-nucleus}, with high probability the $1$s transmit from a $1$-nucleus along a highway to the origin by time 
\[
2p^4 M \cdot 4 ((1/p)\log(1/p))^2 \ll pM,
\]
which is before time $M$. Therefore, if the origin is initially in state $0$, then it will flip to state $1$ by time $M$ with high probability, and thus will not be in state $0$ in the final configuration.
\end{proof}

\section{Occupation by 2s in the standard variant with $q$ small}
\label{sec:occupation-2-standard}

In this section, we show that the origin will likely eventually be occupied by a $2$ when $1$s and $2$s both grow according to the standard threshold-$2$ dynamics (Theorem \ref{thm:2swininstandard}).

We make use of boxes at three scales. The $p$-boxes were already defined in
Section~\ref{sec:eliminate-0s} but we repeat the definition here; the other two 
types of boxes are introduced now:
\begin{itemize}
    \item \textbf{$p$-boxes} have dimensions $2\floor{\frac{1}{p}\log(1/p)}$ by $2\floor{\frac{1}{p}\log(1/p)}$;
    \item \textbf{$q$-boxes} have dimensions $4\floor{\frac{1}{q}\log(1/q)}$ by $4\floor{\frac{1}{q}\log(1/q)}$ and the \textbf{center box} of a $q$-box is the $2\floor{\frac{1}{q}\log(1/q)}$ by $2\floor{\frac{1}{q}\log(1/q)}$ box with the same center as the $q$-box; and
    \item \textbf{big boxes} have dimensions $2L$ by $2L$, where $L := \floor{e^{\epsilon/q}}\cdot \floor{(1/q)\log(1/q)}$, and the constant $\epsilon$ will be determined later.
\end{itemize}

\subsection{$p$-boxes and $q$-boxes}

This time, we call a $p$-box \textbf{$1$-crossable} if: it contains no $2$s inside or on the external boundary; and every row and every column contains a $1$ initially. 
Note that this definition is slightly different from the one in the previous section.
We again call a $p$-box a \textbf{$1$-nucleus} if it is internally spanned by $1$s. 
A \textbf{frame} is a rectangle such that the number of rows or the number of columns is at least $6$, and each side contains a state-$2$ vertex and there is an additional state-$2$ vertex within distance $2$ of each side also within the rectangle.
A frame is said to be contained in a region if all of these defining state-$2$ vertices are contained in this region.
Every $q$-box is tiled by $p$-boxes, and we call a $q$-box \textbf{fillable} if the following holds:
\begin{itemize}
    \item[(F1)] 
    within the $q$-box, there is a circuit of $1$-crossable $p$-boxes surrounding the center box;
    \item[(F2)] 
    the connected components of $p$-boxes within the finite region bounded by the circuit, which are not connected via $1$-crossable $p$-boxes to the circuit in (F1), each have diameter (measured in terms of $p$-boxes) at most $\log(1/q)$;
    \item[(F3)] 
    within the $q$-box, every square with side length $3\floor{\frac{1}{p}\log(1/p)\log(1/q)}$ contains no frame, and every $5 \times 5$ square contains at most two state-$2$ vertices;
    \item[(F4)] 
    every row and every column of the center box contains a $2$ initially.
\end{itemize}

\begin{figure}[h!]
\label{fig:F1circuit}
\begin{center}
\centering
        \resizebox{0.75\linewidth}{!}{
          \input{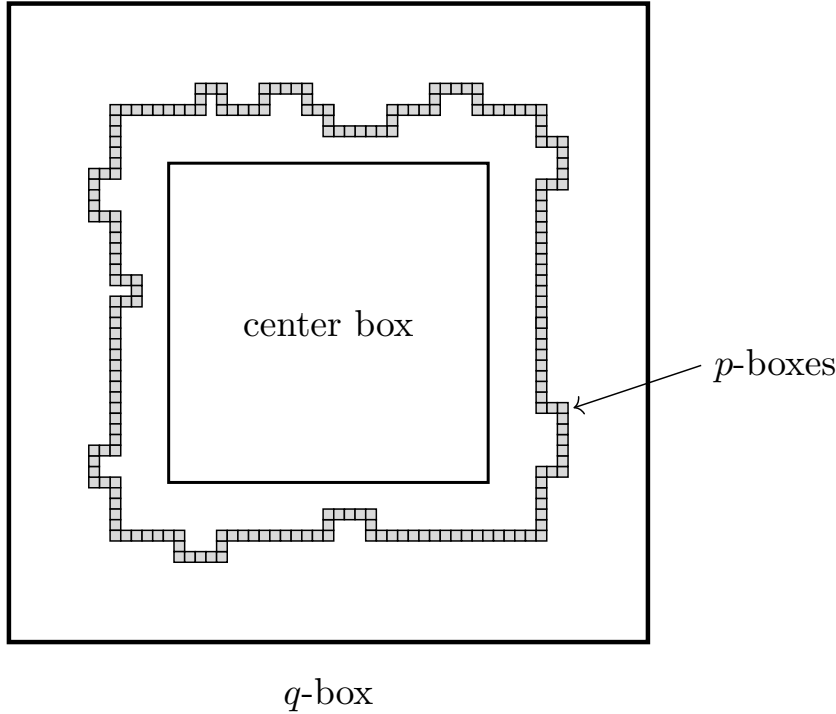}  
        }
\end{center}
\caption{A circuit of $1$-crossable $p$-boxes surrounding the center box in a $q$-box as described in (F1).}
\end{figure}

Under these conditions, if the circuit of $1$-crossable $p$-boxes in a fillable $q$-box gets filled by $1$s at some time, then every $0$ in the center box is eliminated at some time by the following lemma.

\begin{lemma}\label{lem:0elimination}
    Suppose $S\subset \bZ^2$ is a finite region such that in the initial configuration, $S$ does not contain any frame and every $5 \times 5$ square that intersects with $S$ contains at most two state-$2$ vertices. 
    Assume that the external boundary of $S$ is entirely occupied by $1$s from some time $T$ to time $T+|S|$. 
    Then every $0$ in $S$ is eliminated by time $T+|S|$.
\end{lemma}
\begin{proof}
    To bound the time by which all $0$s are eliminated, observe that if no $0$ flips to a $1$ at some step, then no $0$ will flip to a $1$ thereafter. 
    We can therefore equivalently assume that all vertices outside of $S$ are permanently $1$s after time $T$, since $2$s from the outside cannot interfere with the dynamics internal to $S$ before $0$s are eliminated. 
    Without loss of generality, suppose $S$ is connected (otherwise consider the connected components of $S$ separately). 
    Furthermore, by considering the smallest rectangle containing $S$, we may assume that $S$ is a rectangle. 
    
    If $S$ contains at most one $2$ initially, the claim clearly follows.
    Suppose there are initially exactly two $2$s in $S$.
    Then the region outside the smallest rectangle, $R$, containing these two $2$s will eventually be filled by $1$s. 
    Note that $R$ is possibly degenerate as a line segment.
    There are three cases to consider with regard to the dimensions of $R$.

    If $R$ is $1 \times k$, then each of the vertices in $R$ that are initially in state $0$ will eventually have two neighbors in state $1$ (to the north and south), and therefore flip to state $1$. 
    Note that if $k=3$, and the midpoint is in $S$, then this $1$ would subsequently flip to a $2$, but nonetheless the $0$ was eliminated.

    If $R$ is a $2 \times 2$ square, then each corner of $R$ that is initially in state $0$ will eventually have two neighbors in state $1$, and will flip to state $1$. 
    Again, if these $1$s are in $S$, they will then flip to $2$s.

    If $R$ strictly contains a $2 \times 2$ square, then again, every initial $0$ in $R$ eventually flips to a $1$ (and stays in state $1$).

    Now suppose that $S$ contains at least three $2$s initially.
    Since the external boundary of $S$ is all in state $1$ permanently after time $T$, no vertex on the external boundary of $S$ is ever in state $2$.
    Then the assumptions on $S$ implies that $S$ contains no frame at any time.

    At time $T$, we start by considering the topmost, bottommost, leftmost, and rightmost $2$s within the region $S$.
    The four lines hosting these $2$s form a rectangle, and all the $0$s in $S$ outside of this rectangle will eventually be flipped to $1$s.
    If this rectangle is contained in a $5 \times 5$ square, then by assumption there are at most two initial $2$s in this rectangle (note that some of the defining $2$s of this rectangle must coincide), and the external boundary of this rectangle will all become $1$s at some time. 
    The previous base case argument for exactly one or two initial $2$s gives the desired result. 
    
    If one side of this rectangle contains at least $6$ vertices, then since $S$ contains no frame, one of the four sides of this rectangle does not see any additional state-$2$ vertices within distance $2$.
    The state-$2$ vertex living on this particular side therefore cannot block the growth of $1$s from outside, and all the state-$0$ vertices in $S$ within distance $2$ to this side of the rectangle will eventually be flipped to state-$1$.
    Then, we push this side of the rectangle inward (by at least $3$ units) until we find another state-$2$ vertex in $S$.
    If this push makes another side of the rectangle free of $2$s (because these two sides shared a $2$ at the corner), we also push this side inward until we find another $2$ in $S$.
    
    The rectangle has shrunk, and all the $0$s in $S$ outside of this rectangle will eventually be flipped to $1$s.
    Repeat this procedure until the rectangle is contained in a $5 \times 5$ square, and the proof is complete with the base case argument.
\end{proof}

Note that Lemma \ref{lem:0elimination} only works when $1$s perform the standard bootstrap percolation on $0$s.
The argument fails under the modified rule.
For example, observe that if initially a $0$ neighbors $2$s on both its left and right, then the $0$ within this ``$202$" segment will remain in state $0$ forever under the modified rule. Therefore, even the base case (with two $2$s) when $R$ is a $1\times 3$ rectangle fails.
This observation is crucial when we prove that $2$s do not prevail in the modified dynamics.

\begin{lemma}\label{lem:p-boxcrossable}
    A $p$-box is $1$-crossable with high probability provided that $q \ll \frac{p^2}{(\log(1/p))^2}$.
\end{lemma}
\begin{proof}
    The proof is similar to the proof of Lemma \ref{lem:p-boxcrossablewhp}, so we omit it.
\end{proof}

We cover the plane with overlapping $q$-boxes such that the plane is tiled by the center boxes of the $q$-boxes, and we assume that one such $q$-box is centered at the origin.
Call a $q$-box \textbf{good} if it is fillable and all of the $q$-boxes that overlap with this $q$-box are also fillable. 

\begin{prop}\label{prop:q-boxgood}
    A $q$-box is good with high probability provided that $q \ll \frac{p^2}{(\log(1/p))^2}$.
\end{prop}

Proposition\ref{prop:q-boxgood} will follow from the next three lemmas, which show that conditions (F1)--(F4) of a $q$-box to be fillable occur with high probability. 

\begin{lemma}\label{lem:F1F2}
    A $q$-box satisfies conditions (F1) and (F2) with high probability provided that $q \ll \frac{p^2}{(\log(1/p))^2}$.
\end{lemma}
\begin{proof}
    For a fixed $q$-box, if (F1) fails, then there exists a dual path consisting of non-$1$-crossable $p$-boxes from the external boundary of the center box to the boundary of the whole $q$-box.
    Note that this path may take diagonal steps.
    Such a path has length at least 
    \[
    \frac{\floor{(1/q)\log(1/q)}}{2\floor{(1/p)\log(1/p)}} \geq \frac{p\log(1/q)}{3q\log(1/p)}.
    \]
    From any initial location, the number of paths (allowing diagonal steps) of length $j$ is at most $8\cdot 7^{j - 1}$.
    Furthermore, the $p$-boxes are $1$-dependent in the initial configuration, and along one side of the external boundary of the center box there are at most 
    \[
    \frac{4\floor{(1/q)\log(1/q)}}{2\floor{(1/p)\log(1/p)}} + 3 \leq 
    \frac{3p\log(1/q)}{q\log(1/p)}
    \]
    tiled $p$-boxes.
    Set $r := \frac{p\log(1/q)} {q\log(1/p)}$ for convenience, and note that $r \to \infty$ as $p \to 0$.
    By Lemma \ref{lem:p-boxcrossable}, we may choose an arbitrarily small constant $\delta > 0$, and the probability that a $p$-box is not $1$-crossable is at most $\delta$ for $p$ sufficiently small and $q \ll \frac{p^2}{(\log(1/p))^2}$.
    Choose $\delta < 7^{-10}$.
    It follows that the expected number of these dual paths is at most
    \[
    12r \cdot \sum_{j \geq r/3}   
    8\cdot 7^{j - 1} \cdot \delta^{j/10} 
    = \cO(r (7\delta^{1/10})^{r/3})
    \to 0
    \]
    as $p \to 0$.

    Now for (F2), if such a connected component exists with diameter more than $\log(1/q)$, then there is a circuit and thus a path (allowing diagonal steps) of non-$1$-crossable $p$-boxes with length at least $\log(1/q)$ surrounding this component.
    Using a similar argument to the above with $\delta < (7e^3)^{-10}$, we have that the expected number of these paths is at most
    \[
    \left(\frac{4}{q} \log\left(\frac{1}{q}\right)\right)^2 \sum_{j \geq \log (1/q)}  8\cdot 7^{j - 1}\cdot \delta^{j/10} 
    = \cO((1/q^2) (\log(1/q))^2 (e^{-3})^{\log(1/q)}) \to 0
    \]
    as $p \to 0$.
\end{proof}

\begin{lemma}\label{lem:F3}
    A $q$-box satisfies condition (F3) with high probability provided that $q \le p^2$.
\end{lemma}
\begin{proof}
    First, the expected number of $5\times 5$ squares within the $q$-box that contain at least three initial $2$s is at most a constant times 
    \[
    \left((1/q)\log(1/q)\right)^2 \cdot q^3
    = q\cdot (\log(1/q))^2\to 0.
    \]

    Now cover the $q$-box by squares of side length $6 \floor{\frac{1}{p}\log(1/p)\log(1/q)}$, which half overlap with neighboring squares. 
    The number of such squares required to cover the $q$-box is at most
    a constant times 
    \begin{equation}\label{eqn:F3-squares}
    \left(\frac{(1/q)\log(1/q)}{(1/p)\log(1/p)\log(1/q)}\right)^2 = \frac{1}{(\log(1/p))^2} \cdot \frac{p^2}{q^2}.
    \end{equation}

    If the longest side of a frame has at least 6 sites, then there are two parallel lines at distance at least 5, each with two state-$2$ sites within distance 2 of it. In addition, there is a perpendicular line that has within distance 2 either: 
    two of these four $2$s;
    or one of these four 2s and an additional $2$; or two additional $2$s. (If the first two lines are vertical, then the third line may be, say, the top line of the frame, which must have two $2$s within distance $2$.) In order, the three scenarios have their probabilities bounded by $(\log(1/p)\log(1/q))^{10}$ times: 
    \begin{equation*}
        \begin{aligned}
         &\frac{p^2}{q^2} \cdot \frac{1}{p^2} \cdot  \frac{q^4}{p^3}=
    \frac{q^2}{p^3};\\
    &
    \frac{p^2}{q^2} \cdot \frac{1}{p^2} \cdot  \frac{q^5}{p^5}=
    \frac{q^3}{p^5}; \text{ and}
    \\&
    \frac{p^2}{q^2} \cdot \frac{1}{p^3} \cdot  \frac{q^6}{p^6}=
    \frac{q^4}{p^7}.   
        \end{aligned}
    \end{equation*}
   For example, the first factor in the second probability bound 
   represents the number of squares needed to cover the $q$-box bounded by (\ref{eqn:F3-squares}) and the second factor bounds the number of lines within a fixed square. Finally, for the third factor, we first choose the four $2$s on the two original lines, which gives the factor $(q/p)^4$, and then an additional $2$ on a perpendicular line through one of the chosen four locations, which gives the additional factor $q/p$. 
    As $q\le p^2$, the three probabilities approach $0$ as $p\to 0$.   
\end{proof}

\begin{lemma}\label{lem:F4}
    A $q$-box satisfies condition (F4) with high probability as $q \to 0$.
\end{lemma}
\begin{proof}
    The probability that a fixed $q$-box does not satisfy (F4) is at most
    \[
    2\left(\frac{2}{q}\log\left(\frac{1}{q}\right)\right) (1 - q)^{\frac{2}{q}\log(1/q) - 2} \to 0
    \]
    as $q \to 0$.
\end{proof}

\begin{proof}[Proof of Proposition \ref{prop:q-boxgood}]
    For any fixed $q$-box $B$, there are at most $24$ other $q$-boxes that share a common vertex or neighboring vertex with $B$.
    When $q \ll \frac{p^2}{(\log(1/p))^2}$, each of these $q$-boxes (at most $25$ in total) is fillable with high probability by Lemmas \ref{lem:F1F2}, \ref{lem:F3}, and \ref{lem:F4}.
    Hence the probability that $B$ is not good tends to $0$ as $p \to 0$ and $q \ll \frac{p^2}{(\log(1/p))^2}$ by the union bound. 
\end{proof}

\subsection{Big boxes}

Recall that a big box has dimensions $2L$ by $2L$ with $L = \floor{e^{\epsilon/q}}\cdot \floor{(1/q)\log(1/q)}$.
(The extra factor ensures exact partition of a big box with center boxes of $q$-boxes.)
Consider the four sub-rectangles (two vertical and two horizontal) of dimensions $2L$ by $L - 3\floor{q^{-2}}$ and $L - 3\floor{q^{-2}}$ by $2L$ along the boundary of the big box, each of whose long side is $3\floor{q^{-2}}$ away from the boundary of the big box. (We need this correction to prevent interference with the ignition of $2$s; see Lemma~\ref{lem:spreadtoorigin}).
We call a big box \textbf{successful} if the following conditions are satisfied:
\begin{enumerate}
    \item[(S1)] 
    each of the four sub-rectangles are crossed the long way by paths of good $q$-boxes (see Figure \ref{fig:S1});
    
    \item[(S2)] 
   from each good $q$-box $Q$ in the set of crossings given in (S1), along the same roughly horizontal (or vertical) path that contains $Q$ one can find a good $q$-box that contains a $1$-nucleus in its circuit as in (F1) within $\floor{q^4 L}$ $q$-boxes along the path; and
    
    \item[(S3)] 
    the bootstrap percolation of $2$s internal to the region within  $\ell^\infty$-distance $L$ of the big box produces no connected component of $2$s with diameter larger than $1/q$.
\end{enumerate}

\begin{figure}[h!]
\label{fig:S1}
    \centering
        \resizebox{0.75\linewidth}{!}{

\begin{tikzpicture}[scale=0.15,>=Stealth]
  \def\qside{1.0}        
  \def\L{20}             
  \def\buf{3.5}          
  \def\amp{3}            
  \def\seedVleft{31415}  
  \def\seedVright{27184}
  \def\seedHbot{14146}
  \def\seedHtop{17323}
  \pgfmathsetmacro{\BigSide}{2*\L + 1}
  \pgfmathsetmacro{\step}{\qside/2}             
  \pgfmathsetmacro{\quarter}{\qside/4}          

  \draw[thick] (0,0) rectangle (\BigSide,\BigSide);
  \node[below] at (\BigSide/2,-1) {big box};

  \draw[dotted,thick] (0,\L) -- (\BigSide,\L);
  \draw[dotted,thick] (0,\L+1) -- (\BigSide,\L+1);
  \draw[dotted,thick] (\L,0) -- (\L,\BigSide);
  \draw[dotted,thick] (\L+1,0) -- (\L+1,\BigSide);

  \draw[dotted,thick] (0,\buf) -- (\BigSide,\buf);
  \draw[dotted,thick] (0,\BigSide-\buf) -- (\BigSide,\BigSide-\buf);
  \draw[dotted,thick] (\buf,0) -- (\buf,\BigSide);
  \draw[dotted,thick] (\BigSide-\buf,0) -- (\BigSide-\buf,\BigSide);

  \node[left]  at (0.5, \BigSide - 1.5) {$3\lfloor q^{-2}\rfloor$};
  \node[left]  at (0.5, 1.5)            {$3\lfloor q^{-2}\rfloor$};
  \node[above] at (1.25, \BigSide)      {$3\lfloor q^{-2}\rfloor$};
  \node[above] at (\BigSide - 1.25, \BigSide) {$3\lfloor q^{-2}\rfloor$};

  \pgfmathsetmacro{\XL}{\L/2 - 1/4}
  \pgfmathsetmacro{\XR}{\BigSide-\L/2 - 1/4}
  \pgfmathsetmacro{\YB}{\L/2 - 1/4}
  \pgfmathsetmacro{\YT}{\BigSide-\L/2 - 1/4}

  \pgfmathsetmacro{\xminwig}{\buf + \step}
  \pgfmathsetmacro{\xmaxwig}{\BigSide - \buf - \step}

  \pgfmathsetmacro{\ymin}{\quarter + 1/2}              
  \pgfmathsetmacro{\ymax}{\BigSide - \quarter}
  \pgfmathsetmacro{\xmin}{\quarter + 1/2}
  \pgfmathsetmacro{\xmax}{\BigSide - \quarter}

  \newcommand{\qcyan}[2]{%
    \draw[fill=cyan!40,draw=cyan!60!black] (#1-\qside/2,#2-\qside/2) rectangle ++(\qside,\qside);}
  \newcommand{\qorg}[2]{%
    \draw[fill=cyan!40,draw=cyan!60!black] (#1-\qside/2,#2-\qside/2) rectangle ++(\qside,\qside);}

  \newcommand{\VerticalRandom}[2]{
    \pgfmathsetseed{#2}
    \pgfmathtruncatemacro{\T}{int((\ymax-\ymin)/\step)} 
    \pgfmathtruncatemacro{\nmin}{ceil((\xminwig-#1)/\step)}
    \pgfmathtruncatemacro{\nmax}{floor((\xmaxwig-#1)/\step)}
    \def\prevn{0}
    \def\prevsg{0}
    \foreach \t in {0,...,\T}{
      \pgfmathsetmacro{\y}{\ymin + \t*\step}
      \pgfmathtruncatemacro{\u}{floor(100*rnd)}
      \pgfmathtruncatemacro{\sgprop}{ (\u<25 ? -1 : (\u<75 ? 0 : 1)) }
      \pgfmathtruncatemacro{\isCross}{ (abs(\y-\YB)<1e-6) + (abs(\y-\YT)<1e-6) }
      \pgfmathtruncatemacro{\isBounce}{ (abs(\prevsg)==1) && (\sgprop+\prevsg==0) && (\isCross==0) ? 1 : 0 }
      \pgfmathtruncatemacro{\sg}{ (\isBounce==1 ? 0 : \sgprop) }
      \pgfmathtruncatemacro{\ntmp}{min(max(\prevn+\sg,\nmin),\nmax)}
      \pgfmathtruncatemacro{\n}{ (\isCross>0 ? 0 : \ntmp) }
      \pgfmathtruncatemacro{\d}{\n-\prevn}
      \pgfmathtruncatemacro{\kone}{ (abs(\d)>0 ? 1 : 0) }
      \pgfmathtruncatemacro{\sgn}{ (\d>0 ? 1 : (\d<0 ? -1 : 0)) }
      \pgfmathsetmacro{\ybr}{\y - \step}
      \foreach \k in {1,...,\kone}{
        \pgfmathsetmacro{\xk}{#1 + (\prevn+\sgn*\k)*\step}
        \qcyan{\xk}{\ybr}
      }
      \pgfmathsetmacro{\xc}{#1 + \n*\step}
      \qcyan{\xc}{\y}
      \xdef\prevn{\n}
      \xdef\prevsg{\sgn}
    }
  }

  \newcommand{\HorizontalRandom}[2]{
    \pgfmathsetseed{#2}
    \pgfmathtruncatemacro{\T}{int((\xmax-\xmin)/\step)} 
    \pgfmathtruncatemacro{\mmin}{ceil((\buf+\step-#1)/\step)}
    \pgfmathtruncatemacro{\mmax}{floor(((\BigSide-\buf-\step)-#1)/\step)}
    \def\prevm{0}
    \def\prevsg{0}
    \foreach \t in {0,...,\T}{
      \pgfmathsetmacro{\x}{\xmin + \t*\step}
      \pgfmathtruncatemacro{\u}{floor(100*rnd)}
      \pgfmathtruncatemacro{\sgprop}{ (\u<25 ? -1 : (\u<75 ? 0 : 1)) }
      \pgfmathtruncatemacro{\isCross}{ (abs(\x-\XL)<1e-6) + (abs(\x-\XR)<1e-6) }
      \pgfmathtruncatemacro{\isBounce}{ (abs(\prevsg)==1) && (\sgprop+\prevsg==0) && (\isCross==0) ? 1 : 0 }
      \pgfmathtruncatemacro{\sg}{ (\isBounce==1 ? 0 : \sgprop) }
      \pgfmathtruncatemacro{\mtmp}{min(max(\prevm+\sg,\mmin),\mmax)}
      \pgfmathtruncatemacro{\m}{ (\isCross>0 ? 0 : \mtmp) }
      \pgfmathtruncatemacro{\d}{\m-\prevm}
      \pgfmathtruncatemacro{\kone}{ (abs(\d)>0 ? 1 : 0) }
      \pgfmathtruncatemacro{\sgn}{ (\d>0 ? 1 : (\d<0 ? -1 : 0)) }
      \pgfmathsetmacro{\xbr}{\x - \step}
      \foreach \k in {1,...,\kone}{
        \pgfmathsetmacro{\yk}{#1 + (\prevm+\sgn*\k)*\step}
        \qorg{\xbr}{\yk}
      }
      \pgfmathsetmacro{\yc}{#1 + \m*\step}
      \qorg{\x}{\yc}
      \xdef\prevm{\m}
      \xdef\prevsg{\sgn}
    }
  }

  \VerticalRandom{\XL}{\seedVleft}   
  \VerticalRandom{\XR}{\seedVright}  
  \HorizontalRandom{\YB}{\seedHbot}  
  \HorizontalRandom{\YT}{\seedHtop}  

  \coordinate (Label)  at (\BigSide+5,\BigSide/2 + 2);
  \coordinate (Target) at (\XR,\YB);
  \draw[->,thick] (Label) -- (Target);
  \node[right] at (Label) {$q$-boxes};
\end{tikzpicture}  
        }
    \caption{The sub-rectangles and the crossings of good $q$-boxes described in (S1). The paths follow nearest-neighbor adjacency of $q$-boxes. We reserve a buffer zone of width $3\floor{q^{-2}}$ along the big box's boundary for later use.}
\end{figure}
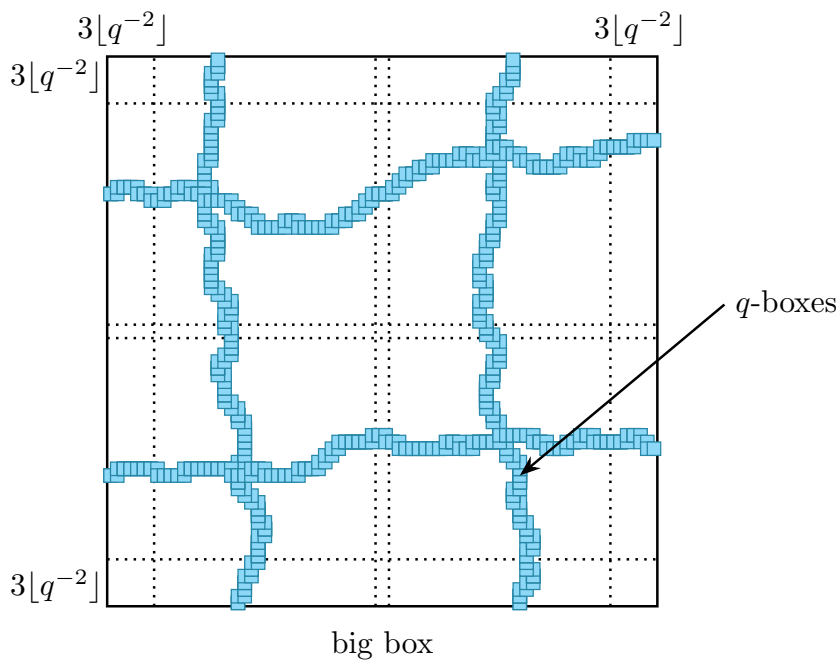


The next lemma demonstrates that 
the conditions ensure that the $1$-nuclei are not too far from the target $q$-boxes so that the spread of $1$s is complete before $2$s can intervene. Then we will show that these conditions hold with high probability, with arguments analogous to those for 
Lemmas~\ref{lem:highway}, \ref{lem:find1-nucleus}, and \ref{lem:nointerventionof2s}.

\begin{prop}\label{prop:bigboxfilling}
    If a big box is successful, then the center box of every good $q$-box in each of the four crossings from (S1) has all of its $0$s eliminated by time $qL$ for any sufficiently small $q$.
\end{prop}
\begin{proof}

    

    First, observe that by (S3) and the fact that every good $q$-box is surrounded by fillable $q$-boxes, the only $2$s in each good $q$-box through time $L$ are those that were initially present or adjacent to an initial $2$ (by (F3)). 
    It follows that every $1$-crossable $p$-box in the circuit from (F1) of a good $q$-box contains no $2$s through time $L$. Note here that the reason we require no initial 2s on the external boundary of a $1$-crossable $p$-box is to prevent invasion of $2$s at the corners due to the $q$-box containing diagonally adjacent $2$s.

    Now the circuits of $1$-crossable $p$-boxes in the good $q$-boxes of a crossing of the big box form a connected set of $1$-crossable $p$-boxes.
    Therefore by (S2), they, together with all of the $1$-crossable $p$-boxes that are connected to them within the good $q$-boxes, are all filled by $1$s at time 
    $$
    T := (q^{4}L + 1)\cdot (4(1/q)\log(1/q))^2 \ll qL.
    $$
    These sites remain $1$s through time $L$ by our previous observation.
    
    Finally, consider the center box of one of the good $q$-boxes. 
    By (F2) and (F3), any connected set $S$ of sites that are not in the connected component of $1$-fillable $p$-boxes that are connected to the circuit of (F1) will satisfy the conditions of Lemma~\ref{lem:0elimination} with $T = (q^{4}L + 1)\cdot (4(1/q)\log(1/q))^2$. 
    Thus by Lemma~\ref{lem:0elimination}, all $0$s are eliminated from the center box by additional time $|S|\le (4(1/q)\log(1/q))^2$, and the total time $T + |S|$ is smaller than $qL$ for sufficiently small $q$. 
\end{proof}
Note that since Proposition \ref{prop:bigboxfilling} relies on Lemma \ref{lem:0elimination}, it only applies to the standard  variant, and not the modified variant.

\begin{prop}\label{prop:bigboxsuccessful}
    A big box is successful with high probability, provided that $q \ll \frac{p^2}{(\log(1/p))^2}$ and $\epsilon < e^{-4}/16$.
\end{prop}
We prove this proposition using the next three lemmas.
\begin{lemma}\label{lem:S1}
    A big box satisfies (S1) with high probability, provided that $q \ll \frac{p^2}{(\log(1/p))^2}$.
\end{lemma}
\begin{proof}
    Fix a big box.
    By symmetry and the union bound, it suffices to show that with high probability, there exists a path of good $q$-boxes within the $2L$ by $L - 3\floor{q^{-2}}$ sub-rectangle on the left half of this big box that travels from the top boundary to the bottom boundary.
    If such a path of good $q$-boxes from the top boundary to the bottom boundary does not exist, then within this sub-rectangle there must be a path of non-good $q$-boxes that travels from the left boundary to the right boundary.
    This path may take diagonal steps, and has length at least
    \[
    \frac{L - 3\floor{q^{-2}}}{4 \floor{(1/q)\log(1/q)}} \geq \frac{\floor{e^{\epsilon/q}}}{8}.
    \]
    Along the long side of the sub-rectangle, there are at most 
    \[
    \frac{2L}{\floor{(1/q)\log(1/q)}} = 2\floor{e^{\epsilon/q}}
    \]
    tiled half-overlapping $q$-boxes.
    By Proposition \ref{prop:q-boxgood}, we may choose an arbitrarily small constant $\delta > 0$, and the probability that a $q$-box is not good is at most $\delta$ for $p$ sufficiently small and $q \ll \frac{p^2}{(\log(1/p))^2}$.
    Choose $\delta < 7^{-25}$ so that $7\delta^{1/25} < 1$.
    It follows that the expected number of these paths is at most
    \[
    2\floor{e^{\epsilon/q}} \cdot \sum_{j \geq \floor{e^{\epsilon/q}}/8}8 \cdot 7^{j-1}\cdot \delta^{j/25} = \cO(e^{\epsilon/q}\cdot (7\delta^{1/25})^{e^{\epsilon/q}}) \to 0
    \]
    as $p \to 0$.
\end{proof}

\begin{lemma}\label{lem:S2}
    A big box satisfies (S2) with high probability, provided that $q \ll \frac{p^2}{(\log(1/p))^2}$.
\end{lemma}
\begin{proof}   
    Fix a big box.
    Since $q \ll \frac{p^2}{(\log(1/p))^2}$, we may assume that (S1) holds by Lemma \ref{lem:S1}.
    Given the four crossings in (S1), if (S2) fails, then within one of these crossings there is a path consisting of $\floor{q^4L} + 1$ good $q$-boxes,  none of which contains a $1$-nucleus in its circuit of $1$-crossable $p$-boxes (as in (F1)).

     By Lemma \ref{lem:1-nucleus}, the probability that a good $q$-box does not contain any $1$-nucleus in its circuit is bounded above by the probability that a fixed one of its $1$-crossable $p$-boxes is 
    not a $1$-nucleus, which is at most $1-e^{-4/p}$. It follows that the expected number of paths of length $\floor{q^4L} + 1$ within the crossings described above  is at most
    a constant times
    \[
        L^2 \cdot \left(1-e^{-4/p}\right) ^{q^4L/25}
        \leq 
        L^2 \exp\left(-e^{-4/p}q^4L/25\right) 
        \leq
        L^2 \exp\left(-\sqrt{L}\right) \to 0
    \]
    as $p \to 0$.   
\end{proof}

\begin{lemma}\label{lem:S3}
    A big box satisfies (S3) with high probability as $q \to 0$ if $\epsilon < e^{-4}/16$.
\end{lemma}
\begin{proof}
    The proof is similar to Lemma \ref{lem:nointerventionof2s}.
    Fix a big box and consider the larger box with dimensions $5L$ by $5L$ containing the big box at its center.
    In the initial configuration, we flip all the state 0 vertices in the larger box to state 1, and consider the internal dynamics in the larger box.
    All the maximal connected components of $2$s in the resulting final configuration are rectangles. 
    For $j \in \N$, define $E_j$ to be the event that the final configuration contains a rectangle of state-$2$ vertices with the longest side length at least $j$.
    If $E_j$ occurs, then by Lemma \ref{lem:Aizenman-Lebowitz}, the large box of dimensions $5L$ by $5L$ contains a rectangle $R$ that is filled completely by $2$s in the final configuration, and $\operatorname{long}(R) \in [j/2, j]$.
    Consequently, every pair of neighboring lines that intersect $R$ and are perpendicular to the longest side of $R$ contains a state-$2$ vertex within $R$ in the initial configuration.
    Furthermore, two pairs of neighboring parallel lines satisfy this condition independently if they do not overlap.
    The number of such rectangles $R$ within the large box is at most $(5L)^2 j^2$.
    As $1 - q \geq e^{-2q}$ for all small enough $q$, we have 
    \[
    \P(E_j) \leq 25 L^2 j^2 (1 - (1 - q)^{2j})^{j/4 - 1} 
    \leq
    25 L^2 j^2(1 - e^{-4jq})^{j/4 - 1}
    \leq
    25 L^2 j^2\exp(-(j/4 - 1)e^{-4jq})
    .
    \]
    Take $j = \floor{1/q}$ and note that $L \leq e^{2\epsilon/q}$ for all sufficiently small $q$.
    The expression above vanishes as $q \to 0$ if $\epsilon < e^{-4}/16$.
\end{proof}
\begin{proof}[Proof of Proposition \ref{prop:bigboxsuccessful}]
    Combine Lemmas \ref{lem:S1}, \ref{lem:S2}, and \ref{lem:S3} and apply the union bound.
\end{proof}


For the proof of Theorem~\ref{thm:2swininstandard}, we will require the
$1$s to reach the origin with high probability by an appropriate time. 
As the 
origin is a center of a big box, we thus need following stronger property.

We call a big box $B$ \df{centrally successful} if it is successful and 
the following two conditions also hold:
\begin{enumerate}[leftmargin=1cm]
    \item[(CS1)] 
    it contains a path of good $q$-boxes connecting the $q$-box that includes the center of $B$ to the boundary of $B$; and
    
    \item[(CS2)] 
   from each good $q$-box $Q$ on this path, one can find a good $q$-box that contains a $1$-nucleus in its circuit as in (F1) within $\floor{q^4 L}$ $q$-boxes from $Q$ along the path.
\end{enumerate}

In particular, in a centrally successful box, the path from (CS2) intersects the crossings from (S1). 
The following two propositions are then proved in a similar manner as 
Propositions~\ref{prop:bigboxfilling} and~\ref{prop:bigboxsuccessful}.

\begin{prop}\label{prop:bigboxcenterfilling}
    If a big box is successful, then the center box of every good $q$-box in each of the four crossings from (S1) has all of its $0$s eliminated by time $qL$ for any sufficiently small $q$.
\end{prop}

\begin{prop}\label{prop:bigboxcentrallysuccessful}
    A big box is centrally successful with high probability, provided that $q \ll \frac{p^2}{(\log(1/p))^2}$ and $\epsilon$ is small enough.
\end{prop}

\subsection{2-ignition}

So far, we provided a mechanism for elimination of $0$s 
in a big box. To finish the proof of 
Theorem~\ref{thm:2swininstandard}, we need to demonstrate that $2$s eventually prevail. A successful big box $B$ is \textbf{$2$-ignited} if the $2L$ by $2L$ square $B'$ at distance exactly $2\floor{(1/q)\log(1/q)}$ directly below $B$ initially contains no $0$s, every row of $B'$ contains a $2$, and the bottom row of $B'$ contains only $2$s. 
We cover the lattice by overlapping big boxes, where the big box $B_{ij} = [-L + 1,L]^2 + (iL,jL)$, and consider the configuration of successful big boxes in the upper half-plane, that is, for $(i,j)\in \mathbb{Z} \times \mathbb{Z}_{\ge 0}$. 
Consider $B_{00}$, the big box at the origin.
Recall that the plane is tiled 
by the center boxes of the $q$-boxes. 
We may assume that $B_{00}$ is exactly tiled by these center boxes, as $L$ was defined as an integer multiple of $\floor{(1/q)\log(1/q)}$ for this purpose.

\begin{lemma}\label{lem:uniqueinfinitecomponent}
    For sufficiently small $p$ and $q \ll p^2 / (\log (1/p))^2$ the successful big boxes form a unique infinite component (with nearest-neighbor adjacency) in the upper half plane with probability $1$. This component includes $B_{00}$ with high probability as $p\to 0$.
\end{lemma}
\begin{proof}
Note that in the initial configuration, if we mark successful vertices open and unsuccessful vertices closed, then vertices at $\ell^\infty$-distance at least $3$ are marked independently.  Recall that in the 
    standard supercritical independent site percolation on $\bZ^2$, there is a.s.~an infinite connected component contained entirely in the upper 
    half-plane, and also there are a.s.~infinitely many pairs of sites $(-x,0)$, $(x,0)$, $x>0$, that are connected through 
    open paths that are entirely contained in the upper half-plane and disjoint for different $x$.   Thus the first claim holds by Proposition \ref{prop:bigboxsuccessful} and \cite{LiggettSchonmannStacey}. The second claim follows as the probability that a vertex is open goes to 1 as $p\to 0$.
\end{proof}

\begin{lemma}
\label{lem:spreadtoorigin}
    As $p \to 0$ with $q \ll \frac{p^2}{(\log(1/p))^2}$, both of the following occur with high probability:
    \begin{enumerate}
        \item[(i)]
        the big box $B_{00}$ centered at the origin is in the infinite component of successful big boxes that contains a $2$-ignited big box; and

        \item[(ii)]
        the big box $B_{00}$ is centrally successful.   
    \end{enumerate}
\end{lemma}

\begin{proof} Let $G_1$ be the event that $B_{00}$ is in the infinite component in the upper half-plane of successful big boxes. We know from Lemma~\ref{lem:uniqueinfinitecomponent} that 
$\P(G_1) > 0$. 

We will now show that, conditional on $G_1$, the 
probability that the big box $B_{00}$ is $2$-ignited is nonzero. 
To see this, let $G_2$ be the event that the initial configuration of $1$s and $2$s in the $2L$ by $2L$ square 
\[
B_{00}'=[-L+1,L]\times [-3L+2-2\floor{(1/q)\log(1/q)}, -L+1-2\floor{(1/q)\log(1/q)}]
\]
is as follows, illustrated in Figure~\ref{fig:ignition}: 
for each integer
$$j \in [-3L+2-2\floor{(1/q)\log(1/q)}, -L+1-2\floor{(1/q)\log(1/q)}],$$ 
there is a $2$ at location
\[
\begin{cases}
    (-L+1+\floor{1/q^2},j) & \text{if } j \equiv 0 \mod 4 \\
    (L- \floor{1/q^2},j)   & \text{if } j \equiv 1 \mod 4 \\
    (-L+2+\floor{1/q^2},j) & \text{if } j \equiv 2 \mod 4 \\
    (L- 1 - \floor{1/q^2},j)   & \text{if } j \equiv 3 \mod 4
\end{cases};
\]
the bottom strip $[-L+1,L]\times \{-3L+2-2\floor{(1/q)\log(1/q)}\}$ contains only $2$s; and all other sites are in state $1$. 
This configuration is not mutually exclusive to $B_{00}$ being in an infinite component of successful big boxes because the only property it needs to maintain is that the bootstrap percolation of $2$s within distance $L$ of the big box $B_{00}$ does not internally produce a component of $2$s with diameter larger than $1/q$. 
This is because the part of the configuration in $[-L+1,L]\times [-2L,-L+1-2\floor{(1/q)\log(1/q)}]$ is inert and the nearest $2$s to the left or right boundaries are more than distance $1/q^2$ away from the $2$s in the square, so growth of $2$s cannot enter from the sides.

Furthermore, with high probability the $2$ in the top row of $B_{00}'$
is at distance at least 3 from any $2$ in the final configuration of 
bootstrap percolation of $2$s within distance $L$ of $B_{00}$. 
In addition, the four columns hosting the $2$s in this square have horizontal distance at least $1/q^2$ to the good $q$-boxes (and the fillable $q$-boxes around them) in the crossings of $B_{00}$. The presence of additional $1$s in $B_{00}'$ does not interfere with the requirements (F1)--(F3) for fillable $q$-boxes.  

 
It follows that $P(G_2\mid G_1)>0$ and thus $\P(G_1\cap G_2) > 0$.

The event that there exists $i\in \bZ$ such that the big box at $B_{i0}$ is in the infinite component of successful big boxes and is $2$-ignited is horizontally translation invariant (under translations by $L$ and with respect to the product measure), and therefore occurs with probability $1$ by the Ergodic Theorem. This proves part (i). Part (ii) 
follows from Proposition~\ref{prop:bigboxcentrallysuccessful}.
\end{proof}

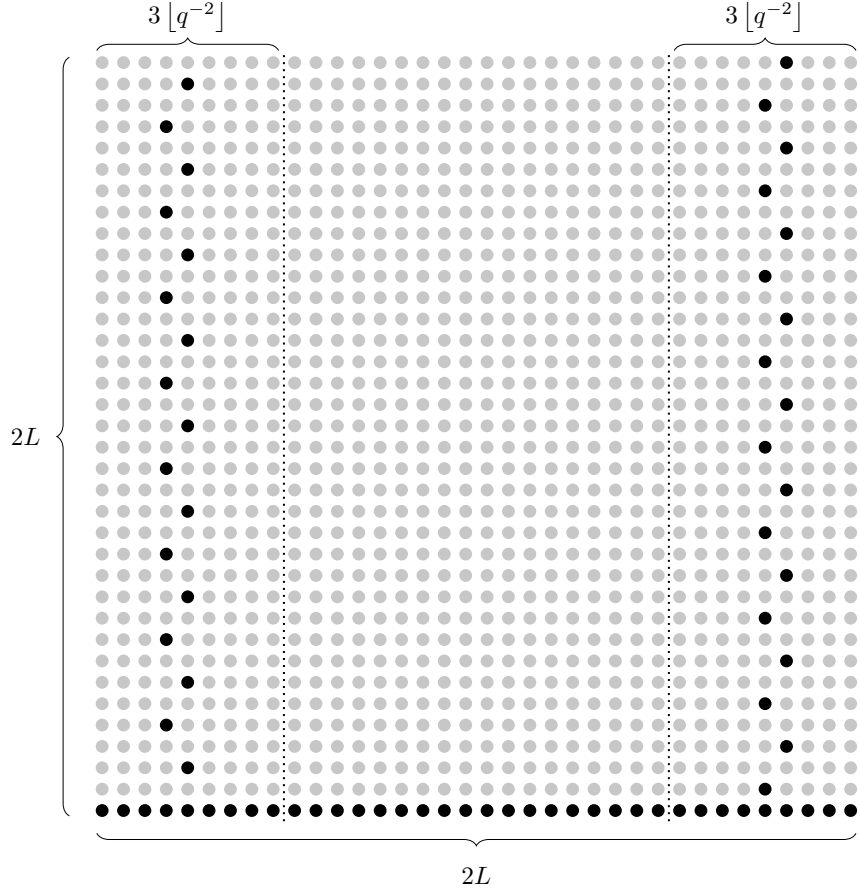
\begin{figure}[h!]
\label{fig:ignition}
\begin{center}
\centering
        \resizebox{0.75\linewidth}{!}{

\begin{tikzpicture}[scale=0.30]
    \foreach \row in {1,...,35} {
        \pgfmathtruncatemacro{\m}{mod(\row,4)}
        \foreach \col in {0,...,35} {
            \def\colmod{mgrey}
            \ifnum\m=0
                \ifnum\col=3 \def\colmod{black}\fi
            \fi
            \ifnum\m=2
                \ifnum\col=4 \def\colmod{black}\fi
            \fi
            \ifnum\m=1
                \ifnum\col=31 \def\colmod{black}\fi
            \fi
            \ifnum\m=3
                \ifnum\col=32 \def\colmod{black}\fi
            \fi
            \fill[\colmod] (\col,\row) circle (0.3);
        }
    }

    \foreach \col in {0,...,35} {
        \fill[black] (\col,0) circle (0.3);
    }

\draw[decorate,decoration={brace,amplitude=5pt,raise=-8pt}]
  (-0.25,36.5) -- (8.25,36.5) node[midway,above=-4pt]{$3\floor{q^{-2}}$};

\draw[decorate,decoration={brace,amplitude=5pt,raise=-8pt}]
  (26.75,36.5) -- (35.25,36.5) node[midway,above=-4pt]{$3\floor{q^{-2}}$};

\draw[decorate,decoration={brace,amplitude=5pt,raise=-4pt}]
  (-2,-0.25) -- (-2,35.25) node[midway,left=4pt]{$2L$};

\draw[decorate,decoration={brace,mirror,amplitude=5pt,raise=-8pt}]
  (-0.25,-2) -- (35.25,-2) node[midway,below=2pt]{$2L$};

\draw[dotted, thick] (8.5, -0.5) -- (8.5, 35.5);
\draw[dotted, thick] (26.5, -0.5) -- (26.5, 35.5);

\end{tikzpicture}  
        }
\end{center}
\caption{Illustration of the configuration in $B_{00}'$ leading to $2$-ignition of the big box $B_{00}$ above. Black dots are $2$s and gray dots are $1$s. Except in the very bottom row, the $1$s are placed in the previously reserved buffer zone so that they do not interfere with the requirements for the good $q$-boxes in the crossings from $B_{00}$ above.}
\end{figure}

To ensure the spread of $2$s through the center boxes in the crossings, we 
will use following monotonicity lemma.

\begin{lemma}\label{lem:remove0s}
    Consider two dynamics with the same initial configurations of $1$s and $2$s on $\bZ^2$. 
    The first dynamics is the standard one on the full $\bZ^2$, and denote by $A_f\subset \bZ^2$ the set of vertices that are eventually occupied by $2$s in these dynamics. 
    For the second dynamics, we remove from $\bZ^2$ all the vertices that are initially $0$s, then run the standard dynamics on the induced subset of $\bZ^2$ and obtain the set $A_r$ of vertices that are eventually in state $2$. 
    Then $A_r\subset A_f$.
\end{lemma}
\begin{proof}
We will prove the following stronger claim.
Let $A_f(t)$ and $A_r(t)$ be the sets of vertices in state $2$ at time $t$ for the two dynamics. 
Then $A_r(t)\subset A_f(t)$ for all $t\ge 0$. We prove this by 
induction on $t$. The case $t=0$ holds by definition. Assume now that $A_r(t-1)\subset A_f(t-1)$. Then $A_r(t-1)\subset A_f(t)$, so we only need 
to show that $A_r(t)\setminus A_r(t-1)\subset A_f(t)$. 
Take any vertex $x\in A_r(t)\setminus A_r(t-1)$.  Then, at time $t-1$, $x$ is in state $1$ in the second dynamics (with removed $0$ sites) and, by the induction hypothesis, has two neighbors in state 
$2$ in both dynamics. In particular, $x$ was not removed for the second dynamics and hence it cannot be in the state $0$ even at time $t=0$ in the first dynamics. Thus, in the first dynamics, $x$ is at time $t-1$ either in state 2 or in state 1, and in either case $x\in A_f(t)$.
\end{proof}

We now identify the mechanism that makes $2$s are eventually widespread.
\begin{lemma}\label{lem:spreadof2}  
    Fix a big box $B$. Assume that $B$ is successful and that it is connected by a path of successful big boxes to a $2$-ignited successful big box. Then $B$ has the property that the center box of every $q$-box in all four of the crossings of $B$ by good $q$-boxes  will eventually be completely occupied by $2$s. If, in addition, $B$ is centrally successful, then the center boxes in its path guaranteed by (CS1)--(CS2) is also completely occupied by $2$s.
\end{lemma}
\begin{proof}
Proposition \ref{prop:bigboxfilling} guarantees that all the $0$s in the center boxes of the good $q$-boxes in the crossings of the successful big boxes are eliminated by time $L$. Start with the configuration at time $L$ and remove all the sites that are at state $0$ at this time. 
    The resulting dynamics is then a 
    bootstrap percolation of $2$s on $1$s 
    on a subset of $\bZ^2$, with no $0$s. By Lemma \ref{lem:remove0s}, 
    it is enough to demonstrate the claim for these dynamics.

    The occupation by $2$s is initiated 
    from the $2$-ignition region, which is eventually 
    entirely occupied by $2$s. Then, the $2$s occupy the center boxes of the good $q$-boxes in the vertical crossings of the $2$-ignited successful big box just above the ignition region. The same now 
    happens for the horizontal crossings. From now on, the connectivity 
    of big boxes transmits occupation of $2$s to the crossings of one 
    big box to the neighboring one. Eventually, this occupation will 
    reach the successful box $B$, and if it is also centrally 
    successful, the center boxes in the crossings of the $q$-boxes in its path will also be completely occupied by $2$s.   
\end{proof}

\begin{proof}[Proof of Theorem \ref{thm:2swininstandard}]
By Lemma~\ref{lem:spreadtoorigin} (ii),
$B_{00}$ is centrally successful with high probability. Then 
Lemma~\ref{lem:spreadof2} and Lemma~\ref{lem:spreadtoorigin} (i) imply that, 
with high probability, the center boxes in the crossings of $B_{00}$ will be completely filled by $2$s, and so will the origin.
\end{proof}

\section{Preventing the spread of $2$s in the modified variant}
\label{sec:preventspreadof2s}

    
    




This section is dedicated to showing that in the modified variant, the origin will never become a $2$ with high probability, as $p$ and $q$ approach $0$ with \textit{any} relative rate.
We will prove the following theorem which is a stronger version of Theorem \ref{thm:2sloseinmodified}.

\begin{theorem}\label{thm:2sloseinmodified-stronger}
    Consider the modified variant of the two-stage bootstrap percolation on $\Z^2$.
    When $p$ and $q$ are both sufficiently small, for some constant $c > 0$, the following holds with probability at least $1 - cq^{41}$.
    In the final configuration of the dynamics, the origin is either not in state $2$, or is contained in a maximal connected set of state-$2$ vertices that has $\ell^\infty$-diameter at most $750$.
\end{theorem}
\begin{proof}[Proof of Theorem \ref{thm:2sloseinmodified} assuming Theorem \ref{thm:2sloseinmodified-stronger}]
    Let $\texttt{Box} := \{x \in \bZ^2 : \|x\|_\infty \leq 750\}$, and $\texttt{Comp}$ be the maximal connected set of state-$2$ vertices that contains the origin in the final configuration (if the origin is never in state $2$, set $\texttt{Comp} = \emptyset$).
    
    We claim that if the origin is eventually in state $2$ and $\texttt{Comp} \subset \texttt{Box}$, 
    then in the initial configuration, $\texttt{Box}$ contains at least one state-$2$ vertex.
    Indeed, if initially there is no $2$ in $\texttt{Box}$, then since the origin is eventually in state $2$, the spread of $2$s to the origin must come from outside $\texttt{Box}$, which results in a connected component of $2$s that joins the origin with the exterior of $\texttt{Box}$, violating the assumption $\texttt{Comp} \subset \texttt{Box}$.

    Now by Theorem \ref{thm:2sloseinmodified-stronger}, it suffices to show that
    \[
    \P(\text{the origin is eventually a 2}, \  \texttt{Comp} \subset \texttt{Box}) \to 0 
    \]
    as $q \to 0$.
    By the claim above, this probability is at most
    \[
    \P(\texttt{Box} \  \text{contains at least one $2$ initially}) \\
    =
    1 - (1 - q)^{|\texttt{Box}|}
     \to 0
    \]
    as $q \to 0$.
\end{proof}

Observe that the above proof actually establishes a stronger statement than Theorem~\ref{thm:2sloseinmodified}: there exists a $p_0>0$ so that 
$$
\lim_{q\to 0}\inf_{p\le p_0}
\P(\text{the origin is never in state $2$})=1.
$$

\subsection{Defining a protected region}

Consider the scenario where two non-neighboring vertices $u, v \in \bZ^2$ in the same row are both in state $2$, and all vertices between $u, v$ in this row are all in state $0$.
We call these state-0 vertices between $u, v$ \textbf{blocking 0}s.
Note that since 1s perform modified threshold-2 bootstrap percolation on 0s, all the blocking 0s will remain in state 0 forever.
This phenomenon still holds if these ``$20\cdots 02$" vertices are in the same column, but for convenience, our definition requires them to be on the same row only.  

Given a nonempty proper subset $A \subset \bZ^2$ and a vertex $u \in \bZ^2$, denote by $\texttt{Nbrs}(u, A)$ the number of the four neighbors of $u$ that are contained in the set $A$.
Given a fixed integer $m \geq 1$ and a configuration $\xi_0 \in \{0,1,2\}^{\bZ^2}$, consider a finite nonempty set $Z \subset \bZ^2$ such that
\begin{enumerate}[leftmargin=1.25cm]
    \item[(PR1)]
    if a vertex $u \in Z$ has $\texttt{Nbrs}(u, Z^c) \geq 2$, then $\texttt{Nbrs}(u, Z^c) - 1$ of these neighbors in $Z^c$ are blocking 0s;
    
    \item[(PR2)]
    for each $u \in Z$ with $\texttt{Nbrs}(u, Z^c) \geq 1$, every vertex $v \in Z$ within $\ell^\infty$-distance $m$ of $u$ has $\xi_0(v) \leq 1$.
\end{enumerate}
Given such a set $Z$ and initial configuration $\xi_0$,
define the internal bootstrap percolation dynamics 
$(\xi_t^Z)_{t \geq 0}$ on $Z$ with initial configuration $\xi_0^Z$ as follows:
\begin{itemize}
\item for $x \in Z$, if $\xi_0(x) = 0$, then $\xi_0^Z(x) = 1$; otherwise $\xi_0^Z(x) = \xi_0(x)$; and then
    \item $\xi_t^Z$  is internal dynamics on $Z$ using this altered initial configuration.
\end{itemize}
The last condition on $Z$ is 
\begin{itemize}[leftmargin=1.25cm]
    \item[(PR3)]
    the final configuration in the dynamics $(\xi_t^Z)_{t \geq 0}$ has no connected set of vertices in state $2$ in $Z$ with $\ell^\infty$-diameter larger than $m/2$.
\end{itemize}

\begin{prop}
\label{prop:restriction}
    Fix an integer $m \geq 1$, a configuration $\xi_0 \in \{0,1,2\}^{\bZ^2}$, and a finite nonempty set $Z \subset \bZ^2$.
    Run two modified bootstrap percolation dynamics:  the first is $\xi_t$, and the second $\xi_t^Z$. 
    Suppose that $Z$ and $\xi_0$ satisfy conditions (PR1)--(PR3).
    Then for every $t \geq 0$ and every $u \in Z$, if $\xi_t(u) = 2$, then $\xi_t^Z(u) = 2$.

\end{prop}
\begin{proof}

Suppose for contradiction that the statement is false.
Let $\tau$ be the first time at which there is a vertex $x \in Z$ such that $\xi_\tau(x) = 2$ and $\xi_\tau^Z(x) = 1$.
Then, by the setup, $\tau > 0$.
At time $\tau - 1$, $x$ has at least two neighbors in state $2$ in the first dynamics but at most one neighbor in state $2$ in the second dynamics. 
By minimality of $\tau$, this disagreement implies $\texttt{Nbrs}(x, Z^c) \geq 1$.
Meanwhile, by (PR2), (PR3), and minimality of $\tau$, any vertex $y \in Z$ with $\texttt{Nbrs}(y, Z^c) \geq 1$ has no neighbors in $Z$ with state $2$ in either dynamics at time $\tau - 1$.
As $x$ has at least two neighbors in state $2$ in $\xi_{\tau - 1}$, all of these state $2$ neighbors of $x$ must be in $Z^c$, and therefore $\texttt{Nbrs}(x, Z^c) \geq 2$.
But by (PR1), since the blocking $0$s remain in state $0$ forever in the first dynamics, $x$ has at most one neighbor in state $2$ (which must be in $Z^c$) in $\xi_{\tau - 1}$, a contradiction.
\end{proof}

Given an initial configuration $\xi_0$, call a set $Z$ a {\bf protected region} if (PR1)--(PR3) all hold.
We will demonstrate that such a suitable $Z$ exists with high probability with an explicit construction.
The constructed set will be within a square of size polynomial in $q^{-1}$.
The next lemma shows that for a set $Z$ of such size, the final configuration of $\xi_t^Z$ does not have any connected component of state-2 vertices of large $\ell^\infty$ diameter in $Z$ (that is, (PR3) holds) with high probability.

\begin{lemma}
    \label{lem:PR3}
    Fix an integer $s \geq 1$ and consider the box $A = [-\lfloor q^{-s} \rfloor, \lfloor q^{-s} \rfloor]^2$.
    Initially, each vertex $x \in A$ is in state $2$ with probability $q$ and is in state $1$ otherwise.
    Let the state-$2$ vertices perform the internal bootstrap percolation on state-$1$ vertices in $A$.
    Then in the final configuration of these dynamics, all the maximal connected sets of state-$2$ vertices have $\ell^\infty$-diameter at most $16s$ with probability at least $1 - c_0q^s$, where $c_0 = c_0(s)$ is a constant.
\end{lemma}
\begin{proof}
    The proof is similar to Lemmas \ref{lem:nointerventionof2s} and \ref{lem:S3}, but here a slightly cruder estimate suffices.
    Again, all the maximal connected sets of state-$2$ vertices in the final configuration are rectangles. 
    For $j \in \N$, again denote by $E_j$ the event that the final configuration contains a rectangle of state-$2$ vertices with the longest side length at least $j$.
    By Lemma \ref{lem:Aizenman-Lebowitz}, when $E_j$ occurs, the box $A$ contains a rectangle $R$ with $\operatorname{long}(R) \in [j/2, j]$ such that $R$ is completely filled by $2$s in the final configuration.
    This means that every pair of neighboring lines that intersect $R$ and are each perpendicular to the longest side of $R$ contains a state-$2$ vertex within $R$ in the initial configuration.
    Once again, two pairs of neighboring parallel lines satisfy this condition independently if they do not overlap.
    Now, the number of such rectangles $R$ within $A$ is at most $(2\lfloor q^{-s} \rfloor + 1)^2 j^2 \leq 5q^{-2s}j^2$.
    Thus,
    \[
    \P(E_j) \leq 5q^{-2s} j^2 (2jq)^{j/4 - 1} 
    \leq
    (2j)^{j/4 + 2} q^{(j - 8s)/4 - 1}.
    \]
    Taking $j = 16s$ completes the proof.
\end{proof}

\subsection{Constructing a shell of helpful boxes}

Fix integers $k, m \geq 1$, whose values are to be determined later.
Let $N := 10\lfloor q^{-1} \rfloor$, and define 
\[
\texttt{Cross}
:= 
\left([-m-1, m+1] \times \left[-N, N\right] \right) \bigcup \left(\left[-N, N\right] \times [-m-1, m+1]\right),
\]

\[
\texttt{Segment}
:=
 [0, k + 1]\times \{0\},
\]
and for each vertex $x \in \bZ^2$, we write $\texttt{Cross}(x) := x + \texttt{Cross}$ and $\texttt{Segment}(x) = x + \texttt{Segment}$.
We call a vertex $x \in \bZ^2$ \textbf{supportive} if the following conditions hold: 
\begin{itemize}[leftmargin=1cm]
    \item [(SV1)]
    $\xi_0(x) = 2$;

    \item [(SV2)]
    There is another vertex $y \in \bZ^2$ ($y \neq x$) to the right of $x$, such that $y$ shares the same row as $x$, $2\le\|y - x\|_1 \leq k + 1$, $\xi_0(y) = 2$, and $\xi_0(z) = 0$ for all vertices $z$ between $x$ and $y$ in this row; and

    \item [(SV3)]
    $\xi_0(y) \neq 2$ for all vertices $y \in \texttt{Cross}(x) \setminus \texttt{Segment}(x)$.
\end{itemize}
Next, for each $u \in \bZ^2$, we define the rescaled box $Q_u$ at $u$ to be 
\[
Q_u := \left(\lfloor q^{-1} \rfloor + 1 \right) u + \left[-\frac{\lfloor q^{-1}\rfloor}{2}, \frac{\lfloor q^{-1}\rfloor}{2}\right]^2.
\]
We say that a box $Q_u$ is \textbf{helpful} if it contains at least one supportive vertex that is at least distance $m + k + 4$ from the internal boundary of the box.
\begin{remark}
\label{rmk:bufferstrip}
    This $m + k + 4$ distance requirement ensures that when two helpful boxes neighbor each other, there is a buffer strip of width at least $2m + 5$ where no requirement is imposed in the initial configuration. 
    This will be used later in the construction of a protected region.
\end{remark}

\begin{lemma}\label{lem:searchinggoodvertices}
    Fix $\epsilon > 0$ and a vertex $u \in \bZ^2$.
    Then one can take $k$ large enough (depending on $m$ and $\epsilon$ but not on $p,q$) so that
    \[
    \P(Q_u \text{ is helpful}) \geq 1 - \epsilon,
    \]
    provided that $p$ and $q$ are sufficiently small.
\end{lemma}

\begin{proof}
Assume first that $M := \lfloor\delta/q\rfloor$, for some small  
$\delta>0$, which will be picked later and depends only on $m$ and $\epsilon$. Our choice of $k$ will depend only on $\delta$ and thus 
also have the same dependence.

Call the box $B := [1,M]\times [1,M]$, or any other box of the same size, {\bf viable\/} if it contains a vertex 
$x$ that satisfies (SV1) and the following condition, weaker than (SV2):
\begin{itemize}[leftmargin=1.25cm]
    \item [(SV2')]
    there is another vertex $y \in \bZ^2$ ($y \neq x$) to the right of $x$, such that $y$ shares the same row as $x$, $2\le\|y - x\|_1 \leq k + 1$ and $\xi_0(y) = 2$.
\end{itemize}

Let 
$I_x$ be the indicator of the event that $x$ satisfies (SV1) and (SV2'), 
for 
\[
x\in [5m + 5k, M-(5m + 5k)]\times[5m + 5k,M-(5m + 5k)].
\]
For $x$ in this range that satisfies (SV1) and (SV2')  with a corresponding $y$, we call $(x,y)$ a {\bf viability pair\/}.

Note that $\E I_x\le kq^2$ for all $x$. Moreover, $I_x$ and $I_w$
are independent unless $x$ and $w$ are on the same horizontal line at distance at most $k+1$. In the latter case, 
assuming $x$ is to the left of $w$, the horizontal interval of $2k+3$ vertices starting at $x$ must contain at least three $2$s. It follows that when $I_x$ and $I_w$
are dependent, 
$$
\E(I_xI_w)\le 5^3 k^3 q^3. 
$$
By the local dependence theorem for the Poisson approximation 
(e.g., Corollary 2.C.5 in \cite{BHJ1992}), 
the total variation distance between the distribution of $W := \sum_x I_x$ and $\text{\rm Poisson}(\E W)$ is bounded above by an absolute constant times
$$
M^2k^4q^3\le \delta^2 k^4 q.
$$
Further, if $qk$ is small enough, 
$$
\E I_z\ge q(1-(1-q)^k)\ge \frac 34 q^2 k
$$
and so 
$$
\E W\ge \frac 12 \delta^2k.
$$
We pick $k=\lceil 1/\delta^2\rceil$, to get 
$$
\P(B\text{ is viable})\ge \P (W\ge 1)\ge 1-e^{-1/2}-\cO(\delta^{-6}q)> 1/3,
$$
if $q$ is small enough. 

Now consider the boxes of size $M\times M$ stacked diagonally without overlap within
$Q_u$, so that there are at least $0.9/\delta$ of them. We call these 
{\bf diagonal boxes\/}. As these boxes 
are viable independently with probability at least $1/3$, 
large deviations for the binomial distribution imply that, for 
$D=\lceil 1/(6\delta)\rceil$ and an absolute constant $c>0$, 
\begin{equation}\label{eq:D-viable}
\P(\text{the
number of viable diagonal 
boxes is at least $D$})\ge 1-\exp(-c/\delta)\ge 1-\epsilon/3, 
\end{equation}
if $\delta$ is small enough.



For a diagonal box $B$, pick an arbitrary order of the (ordered) pairs 
$z=(x,y)$ 
of vertices. Let $F_z$ be the 
event that $z$ is the {\it first\/} viability pair. Then $V_B := \cup_z F_z$ is the event 
that $B$ is viable. Let $G_z$ be the 
event that 
$z=(x,y)$ satisfies 
\begin{itemize}[leftmargin=1.25cm]
\item[(SV3')]
$\xi_0(y) \neq 2$ for all vertices $y \in \texttt{Cross}(x) \setminus \{x,y\}$,
\end{itemize}
and let $A_B := \cup_z(F_z\cap G_z)$. 


Consider any deterministic collection of $D$ diagonal boxes $B_i$, $i=1,\ldots, D$, shorten $V_i := V_{B_i}$ and 
$A_i := A_{B_i}$, and let
$V := \cap_{i=1}^D V_i$ be the event that all of these $D$ boxes are viable.
We proceed to show that $\P(\cup_{i=1}^D A_{i}\mid V)$ is 
$(\epsilon/3)$-close to $1$. 
 
If $i\ne j$, and $z=(x,y)$ and $z'=(x',y')$ are viability pairs in
diagonal boxes $B_i$ and $B_j$, respectively, then 
$\texttt{Cross}(x)$ and $\texttt{Cross}(x')$ intersect in $2(2m+3)^2$
vertices. It follows that 
\begin{equation*}
\begin{aligned}
\P(A_i\cap A_j\mid V)&=\P(A_i\cap A_j\mid V_i\cap V_j)
\\&=
\frac{\sum_{z,z'}\P(F_z\cap F_{z'}\cap G_z\cap G_{z'})}{\sum_{z,z'}\P(F_z\cap F_{z'})}
\\&=
\frac{\sum_{z,z'}\P(F_z\cap G_{z})\P(F_{z'}\cap G_{z'})/(1-q)^{2(2m+3)^2}}{\sum_{z,z'}\P(F_z\cap F_{z'})}
\\&=
\frac{\sum_{z}\P(F_z\cap G_{z})\cdot \sum_{z'}\P(F_{z'}\cap G_{z'})}{\P(V_i)\P(V_j)}\cdot \frac 1{(1-q)^{2(2m+3)^2}}
\\&=
\P(A_i\mid V)\cdot \P(A_j\mid V)
\cdot \frac 1{(1-q)^{2(2m+3)^2}}.
\end{aligned}
\end{equation*}
If $q$ is small enough, the last factor is at most $1+\epsilon/6$. 
Furthermore, 
$$
\P(A_i\mid V)=\P(A_i\mid V_i)
=\frac{\sum_z \P(F_z\cap G_z)}{\sum_z \P(F_z)}.
$$
For a fixed $z=(x,y)$, the events $F_z$ and $G_z$ are both 
decreasing in the configuration of $2$s outside of $\{x,y\}$ and are therefore 
positively correlated, so 
$$ 
\P(F_z\cap G_z)\ge \P(F_z)\,\P(G_z)\ge (1-q)^{|\texttt{Cross}(x_i)|}\, \P(F_z)
\ge  (1-q)^{(2m+3)N}\, \P(F_z)\ge \alpha\,  \P(F_z), 
$$
for a constant $\alpha=\alpha(m)>0$. Consequently, 
$\P(A_i\mid V)\ge \alpha$ and 
$$
\sum_{i=1}^D \P(A_i\mid V)\ge \alpha/(6\delta)\ge 6/\epsilon,
$$
after we choose $\delta=\delta(m, \epsilon)$ small enough. 
Recall that the choice of $\delta$ also fixes $k$. 
By 
Lemma~\ref{lem:2nd-moment}, applied to the conditional probabilities,
\begin{equation}\label{eq:cond-D}
\P(\cup_{i=1}^D A_i\mid V)\ge 1-\epsilon/3.
\end{equation}

We now order {\it all\/} diagonal boxes, say top-down, and again denote 
the resulting sequence by $B_i$, with the same abbreviations
$V_i$ and $A_i$. We also order all $D$-tuples $(i_1,\ldots, i_D)$
of strictly positive integers such that $i_1<\ldots<i_D$, say lexicographically.
Let $H_{i_1,\ldots, i_D}$ be the event that $(i_1,\ldots, i_D)$ is the first $D$-tuple for which the diagonal boxes 
$B_{i_1},\ldots, B_{i_D}$ are all viable. Then, by (\ref{eq:cond-D}) and 
(\ref{eq:D-viable}),
\begin{equation}\label{eq:union-A}
\begin{aligned}
    \P(\cup_i A_i)&\ge \sum_{i_1<\ldots <i_D} \P(\cup_i A_i
\mid H_{i_1,\ldots, i_D})\,\P(H_{i_1,\ldots, i_D})\\
&\ge (1-\epsilon/3) \sum_{i_1<\ldots <i_D}\P(H_{i_1,\ldots, i_D})\\
&= (1-\epsilon/3)\,\P(\text{there are at least $D$ viable diagonal boxes})\\
&\ge (1-\epsilon/3)^2>1-2\epsilon/3.
\end{aligned}
\end{equation}

Finally, 
on $\cup_i A_i$, pick the minimal 
$i$ so that $A_i$ happens, with corresponding first viability 
pair $(x_i,y_i)$. 
Note that this only affects the configuration of $2$s. In order for 
$(x_i,y_i)$ to also satisfy (SV2), we need to additionally ensure 
that none of the (at most 
$k$) vertices between $x_i$ and $y_i$ is initially in 
state $1$, but the conditional probability of this, 
given that $A_i$ happens with a first 
deterministically chosen $(x_i,y_i)$, is 
at least 
$$
[(1-p-q)/(1-q)]^k \ge 1-\epsilon/3, 
$$
if $p$ and $q$ are small enough. Then, by a decomposition argument similar 
to (\ref{eq:union-A}), 
$$
\P(Q_u\text{ is helpful})\ge (1-2\epsilon/3)(1-\epsilon/3)>1-\epsilon,
$$
which ends the proof.
\end{proof}

We now describe a ``protective shell" that will be crucial when we later build the protected region.
This technique is based on the oriented surface construction devised in \cite{LipPerc, GrimmettHolroyd10, GrimmettHolroyd12}.
We start with some definitions that were used in \cite{GS2020, GravnerHolroydSivakoff}.

Given a subset $U \subset \bZ^2$, we say that a vertex $u \in \bZ^2$ off the coordinate axes is \textbf{protected by $U$} provided that
\begin{itemize}
    \item 
    if $u \in [1, \infty)^2 \cup (-\infty, -1]^2$, then both $u + [-2, -1] \times [1, 2]$ and $u + [1, 2] \times [-2, -1]$ intersect $U$; and

    \item 
    if $u \in (-\infty, - 1] \times [1, \infty) \cup [1, \infty) \times (-\infty, - 1]$, then both $u + [-2, -1]^2$ and $u + [1, 2]^2$ intersect $U$.
\end{itemize}
A subset $S \subset \bZ^2$ is called a \textbf{shell} of radius $r \in \N$ when it satisfies the following properties:
\begin{enumerate}[leftmargin=1.25cm]
    \item[(SH1)]
    The set $S$ contains all the vertices $u$ with both $\|u\|_1 = r$ and $\|u\|_\infty \geq r - 3$;

    \item[(SH2)]
    each $u \in S$ satisfies $r \leq \|u\|_1 \leq r + 2\sqrt{r}$ and $\|u\|_\infty \leq r$;

    \item[(SH3)]
    for each $\varphi \in \{(\pm 1, \pm 1)\}$, there exists an integer $\kappa = \kappa(\varphi) \geq r / 2$ such that $\kappa \varphi \in S$; and

    \item[(SH4)]
    if $u = (u_1, u_2) \in S$ and $|u_1| \geq 3, |u_2| \geq 3$, then $u$ is protected by $S$.
\end{enumerate}

Due to the similarity to the method used in \cite{GS2020} and \cite{GravnerHolroydSivakoff}, we omit the proof of the following proposition.
For more details, see Proposition 3.6 of \cite{GS2020} and Proposition 7 of \cite{GravnerHolroydSivakoff}.

\begin{prop}
\label{prop:shellsatisfiesproperties}
    Paint each vertex in $\bZ^2$ independently by black with probability $b$ and white otherwise.
    Let $E_r$ be the event that
    there exists a shell of radius $r$ consisting of only black vertices. 
    There exists $b_1 \in (0, 1)$ such that for every $b > b_1$ and every $r \geq 1$, $\P(E_r) \geq 1/2$.
\end{prop}

\subsection{Constructing a protected region}

We now construct a set $Z \subset \bZ^2$ which will satisfy the conditions for a protected region.

First, suppose that there is a shell $S$ of radius $r$ such that for every $u \in S$, $Q_u$ is a helpful box. 
For each $u = (u_1, u_2) \in S$ such that $|u_1| \geq 3$ and $|u_2| \geq 3$, select a supportive vertex $x_u$ from the box $Q_u$.
For each of these selected $x_u = (x_1, x_2)$, define $x_u' = (x_1', x_2')$ by $x_1' = x_1 + 1$, and $x_2' = x_2 - x_2/|x_2|$.
In other words, if $x_u$ is in the upper (resp. lower) half plane, then $x_u'$ is the lower-right (resp. upper-right) diagonal neighbor of $x_u$.
Let $U$ be the set formed by all these vertices $x_u'$.  
Define a \textbf{fortress} to be a square of side length $11\lfloor q^{-1} \rfloor + 10$ if $\floor{q^{-1}}$ is even and $11\lfloor q^{-1} \rfloor + 11$ if $\floor{q^{-1}}$ is odd, all four of whose corners are supportive vertices.
Suppose that there is such a fortress centered at each of four vertices $(\pm r(\lfloor q^{-1} \rfloor + 1), 0), (0, \pm r(\lfloor q^{-1} \rfloor + 1)$.
For each $v = (v_1,v_2)$ of the $16$ corners of these four fortresses, let $v' = (v_1 + 1, v_2 - v_2/|v_2|)$ be defined similarly as the $x_u'$ from $x_u$ above.
Let $K$ be the set formed by all the $16$ vertices $v'$ defined based on the $16$ corners. 

For each vertex $x \in \bZ^2$, define $\texttt{Rect}(x)$ to be the rectangle with opposite corners at $x$ and the origin. 
For example, if $x = (x_1, x_2)$ has $x_1 > 0$ and $x_2 < 0$, then $\texttt{Rect}(x) = [0, x_1] \times [x_2, 0]$.
Now, we define the set $Z$ by
\begin{equation}
    \label{def:Z}
    Z := \bigcup_{x \in U \cup K} \texttt{Rect}(x).
\end{equation}

\begin{remark}\label{rmk:poscorr} {\rm
    From the definition of $Z$, we can see that, in fact, only the $8$ outer vertices (farther from the origin) from the corners of the fortresses are needed. 
    So we may discard the $8$ inner vertices from the set $K$.
    Meanwhile, given a shell $S$ of radius $r$ each of whose vertices corresponds to a helpful box, the fortress is defined in a way such that the supportive vertices do not interfere with other parts of the construction. 
    More precisely,
    \begin{itemize}
        \item 
        the side length of each fortress is designed to be long enough so that any pair of its supportive vertices at the corners do not interfere with each other in the initial configuration, but also not too long so that the ``arms" still intersect;
        
        \item 
        the $8$ outer supportive vertices from the fortress corners are exactly aligned with the boundaries of the rescaled boxes $Q_u$ in the $\bZ^2$ plane, so that whenever their crosses intersect the boxes, the ``arms" of the crosses go through the buffer strips along the boundaries of the boxes (cf. Remark \ref{rmk:bufferstrip});
        and therefore, these outer supportive vertices for the fortresses do not interfere with the supportive vertices from the helpful boxes associated with $S$ in the initial configuration. 
    \end{itemize}
    }
\end{remark}

\begin{lemma}\label{lem:PR1PR2}
    Suppose that $Z$ is defined above by (\ref{def:Z}).
    Then the set $Z$ satisfies (PR1) and (PR2) as long as $q$ is sufficiently small.
\end{lemma}
\begin{proof}
    By construction, every vertex $z \in Z$ has $\texttt{Nbrs}(z, Z^c) \leq 2$.
    In particular, every $z \in Z$ with $\texttt{Nbrs}(z, Z^c) = 2$ is a convex corner of $Z$ and neighbors a blocking $0$ in $Z^c$ above (resp. below) $z$ if $z$ is in the upper-half (resp. lower-half) plane (by construction, such a corner $z$ cannot be on the $x$-axis for small enough $q$).  
    Thus, $Z$ satisfies (PR1).

    For (PR2), note that the slope of $S$ is locally bounded both above and below due to (SH4), and a similar property holds for the boxes corresponding to the vertices in $S$.
    In particular, the vertices in the shell $S$ form a circuit which takes at most two consecutive steps in the same direction.
    With the fortresses taken into account, from any point on the boundary of $Z$, we can find a convex corner of $Z$ within distance $6/q$ by going up, left, down, or right.
    Such a convex corner diagonally neighbors a supportive vertex, and therefore (PR2) is satisfied. 
\end{proof}
 
Next, we show that a region $Z$, defined by (6.6) and with diameter a power
of $1/q$, 
exists with high probability. For this we need sufficiently many independent trials to find fortresses, which are polynomially unlikely. 
Set $N_0 = 3\lfloor q^{-40}\rfloor, n_0 = \lfloor q^{-21} \rfloor, T = \lfloor q^{-19} \rfloor$, and $\Delta = \lfloor q^{-21} \rfloor$.
For $i = 1,\dots, T$, define the sequence of separated annuli
\[
A_i := \{u \in \bZ^2 : n_0 + (2i - 1)\Delta \leq \|u\|_1 \leq n_0 + 2i\Delta\}.
\]
\begin{lemma}\label{lem:PRexists}
    Fix an integer $m \geq 1$.
    For a small enough $\epsilon > 0$, the following holds provided that $p, q$ are both sufficiently small.
    With probability at least $1 - \exp(-1/(4q^2))$, there exists a protected region $Z$ that satisfies (PR1) and (PR2), contains the origin, and is contained in $\{u \in \bZ^2 : \|u\|_1 \leq N_0\}$.
\end{lemma}
\begin{proof}
    Paint each vertex $u \in \bZ^2$ black if the box $Q_u$ is helpful.
    For $i = 1,\dots, T$, let
    \[
    r_i = \lfloor (n_0 + (2i - 1)\Delta) / (\lfloor q^{-1} \rfloor + 1) \rfloor + 17.
    \]
    Observe that $r_i(\lfloor q^{-1} \rfloor + 1) - 16/q \geq n_0 + (2i - 1)\Delta$, and $\sqrt{r_i} \leq \sqrt{N_0/(\lfloor q^{-1} \rfloor /2 )} \ll \Delta / (\lfloor q^{-1} \rfloor + 1)$ for all $q$ small.
    We are looking for $r_i$ such that there is a shell of radius $r_i$ in which every vertex $x$ corresponds to a helpful box $Q_x$.
    The existence of such a shell of radius $r_i$ depends only on the colors of the vertices in $\{u \in \bZ^2 : r_i \leq \|u\|_1 \leq r_i + 2\sqrt{r_i} \}$, and therefore only on the states of vertices within $A_i$.
    Moreover, notice that vertices $x, y$ with $\|x - y\|_\infty \geq 25$ are painted independently.
    Thus, by \cite{LiggettSchonmannStacey}, the configuration of black vertices dominates a product measure of density $b_1$ (as in Proposition \ref{prop:shellsatisfiesproperties}) provided that $\epsilon > 0$ in Lemma \ref{lem:searchinggoodvertices} is small enough, $k$ is chosen appropriately, and $p, q$ are sufficiently small.
    It follows from Proposition \ref{prop:shellsatisfiesproperties} that in this case, such a desirable shell of radius $r_i$ exists with probability at least $1/2$.
    As discussed in Remark \ref{rmk:poscorr}, the existence of such a shell of radius $r_i$ is independent of the existence of the $8$ outer supportive vertices whose corresponding diagonal neighbors comprise the set $K \subset A_i$ in (\ref{def:Z}).
    As illustrated in the proof of Lemma \ref{lem:searchinggoodvertices}, the probability that a vertex is supportive is at least $c_0 q^2$ for some constant $c_0 > 0$ depending on $m$ and $k$, provided that $p, q$ are sufficiently small. 
    It follows that the $8$ outer vertices are all supportive with probability at least $c_0^{8}q^{16} \gg q^{17}$.
    (In fact, for the holding supportive vertices of the fortresses, we may set $k=1$ in their ``$202$" configurations, so one can take $c_0 = 1/2$.)
    Hence, the set $Z$ defined in (\ref{def:Z}) exists with convex corners $U \cup K \subset A_i$ with probability at least $q^{17}/2$ for each $i$, provided that $p, q$ are sufficiently small.
    Then, since the annuli $A_i$ are separated, the probability that such a set $Z$ does not exist in $A_i$ for all $i = 1,\dots, T$ is at most $(1 - q^{17}/2)^{q^{-19}/2} \leq \exp(-1/(4q^2))$.
    Now, by Lemma \ref{lem:PR1PR2}, if the set $Z$ constructed as in (\ref{def:Z}) exists, then it satisfies conditions (PR1) and (PR2).
    It is clear from the construction that $Z$ contains the origin, and the last claim follows from the observation that $n_0 + 2T\Delta \leq N_0$.
\end{proof}

\begin{proof}[Proof of Theorem \ref{thm:2sloseinmodified-stronger}]
Choose $s = 41$ in Lemma \ref{lem:PR3} so that (PR3) is met.
Then $m = 32s < 1500$.
Combining Proposition \ref{prop:restriction}, Lemma \ref{lem:PR3}, and Lemma \ref{lem:PRexists} completes the proof.
\end{proof}

\section{Discussion and open problems}


\noindent 1. {\it Intermediate phase\/}.
    Theorem~\ref{thm:phase-transition} does not resolve 
    what happens when 
    $p^2/(\log(1/p))^2\ll q \ll p^2$ for the standard variant and when
   $ p^2/(\log(1/p))^2\ll q \ll p^2/\log(1/p)$ for the modified variant. While inconclusive, simulations leave open the possibility of a double phase transition, that is, that 
   $1$s predominate in the final configuration for the standard variant 
   at some range of densities $p$ and $q$. (See the last frame of Figure~\ref{fig:sim}.)
   
   \bigskip
   
\noindent 2. {\it Nucleation of $2$s\/}.
Our method of proving Theorem~\ref{thm:2swininstandard} uses an extraordinarily unlikely construction of a nucleus of $2$s that eventually spreads to the origin. We therefore ask the 
following natural analogue of the classical 
question resolved in \cite{Holroyd2003}. 
Assume that $T_2$ is the first time 
the origin is in state $2$. Assuming $q\ll p^2/(\log(1/p))^2$, 
what is the likely size of $T_2$? In particular, is it true that 
$\P(T_2<\exp(C/q))\to 1$ for sufficiently large $C$? (Clearly, 
$\P(T_2<\exp(c/q))\to 0$ if $c$ is small enough, which can be obtained by switching all initial $0$s to $1$s and using the result in \cite{Holroyd2003}.)

 \bigskip
 
\noindent 3. {\it Cyclic rules\/}.
As mentioned in the Introduction, one may consider other 
update graphs than the oriented linear graph on three states. 
One possibility is the {\it standard  cyclic\/} rule, in which (D2) is replaced by the following: if $\xi_t(x)=2$, then $\xi_{t+1}(x)=0$ if $N_0(x,t)\ge 2$ and $\xi_{t+1}(x)=2$ otherwise. Given the now-symmetric role of three states, the most inviting initial 
state is when they have equal densities, that is, $p=q=1/3$ 
\cite{Fisch1991}. In our context, however, it is also natural to 
consider small $(p,q)$. Simulations are far from conclusive, but 
suggest that $1$s may always dominate 
in the final configuration. 


 \bigskip
 
\noindent 4. {\it Competition rules\/}.
Another possibility is the {\it standard competition rule\/} with the same 
three states $0$, $1$, $2$. In these dynamics,  
$\xi_t(x)=i$ implies $\xi_{t+1}(x)=i$  for 
$i=1,2$, while these two states compete
in case $\xi_t(x)=0$: if $N_i(x,t)\ge 2$
for exactly one $i\in\{1,2\}$, then $\xi_{t+1}(x)=i$, otherwise 
$\xi_{t+1}(x)=0$. 
Such rules were 
introduced in \cite{Gravner1997} as {\it multitype threshold voter models\/}. 
Our methods in fact establish a double phase transition in this, much simpler case. Note the similarity to Theorem 1.1 in \cite{GS24} on another 
competition model.
\begin{theorem}\label{thm:competition}
Assume $(p,q)\to (0,0)$. Then 
\begin{itemize}
\item if $q\ll p^{2}/(\log(1/p)^2$, then $\P(\text{the origin is eventually in state $1$})\to 1$;
\item if $p^2\ll q\ll p^{1/2}$, then $\P(\text{the origin  remains in state $0$ forever})\to 1$; and 
\item if $p^{1/2}\log(1/p)\ll q$, then $\P(\text{the origin is eventually in state $2$})\to 1$.
\end{itemize}
\end{theorem}
\begin{proof}
    The middle claim follows directly from \cite{GravnerMcDonald}, so by symmetry we only need to prove the first claim. Consider the 
    square box $B$ of side length $\lfloor e^{C/p}\rfloor$ around the origin, for a constant $C$ specified below. Call this box
    {\bf winning\/} if the two conditions (W1) and (W2) below are satisfied.  As in Section~\ref{sec:eliminate-0s}, tile $\bZ^2$ by $p$-boxes.

    \begin{itemize}[leftmargin=1cm]
        \item [(W1)]  Replace all $1$s on $B$ by $0$s. Also, turn the state of {\it all\/} 
        sites outside $B$ to $0$ (that is, consider the internal dynamics in $B$). The longest side of any rectangles of $2$ in the final state of the resulting dynamics is at most $1/p$. 
        \item [(W2)]  The $p$-box containing the origin is $1$-crossable and is connected through a 
        path of $1$-crossable boxes of length at most $\exp(C/(3p))$ 
        to a $1$-nucleus.   
    \end{itemize}
Condition (W2) is satisfied with high probability by Lemma~\ref{lem:find1-nucleus} if $C$ is large enough. Condition (W1) is satisfied with high probability for any $C$, as $q\ll p$. Therefore,  
    $B$ is winning with high probability. Finally, in a winning box, the entire $p$-box containing the origin is eventually in state $1$, by 
    the same arguments as in Sections~\ref{sec:eliminate-0s} and~\ref{sec:occupation-2-standard} .
\end{proof}
It is an open problem whether the $\log$ factors in Theorem~\ref{thm:competition} can be removed. 


 \bigskip
 
\noindent 5. {\it General bootstrap rules\/}.
A general framework for two-stage bootstrap dynamics on $\bZ^d$ of the 
type given by (D0)--(D2)
involves two finite families $\cU_1$ and $\cU_2$ of finite subsets 
of $\bZ^d$. Then, each $0$ decides whether to update to a $1$ according 
to the $\cU_1$-bootstrap rule, and simultaneously each $1$ decides whether 
to update to a $2$  according 
to the $\cU_2$-bootstrap rule. Again, $2$s never change. Thus, 
the two families encode an arbitrary monotone solidification cellular automaton for
each of the two transitions
\cite{Bollobs2022, BBMS2022, BBMS22, Ghosh}. 

Assuming $(p,q)\to(0,0)$ along some curve, one can then ask the 
question whether the final density of a state $i\in\{0,1,2\}$ ``wins'' in the sense that its density
approaches $1$. To focus on a 
particular issue, 
we define $\gamma_2=\gamma_2(\cU_1,\cU_2)\in[0,\infty]$ 
to be the infimum over all 
powers $\gamma>0$ such that $q\le p^\gamma$ implies that 
the final density of $2$s approaches $1$ as $p\to 0$. 
In our cases, $\gamma_2=2$ for the standard variant and $\gamma_2=\infty$
for the modified one. We now provide a simple example with $\gamma_2=0$.

If $d=2$ and $\cU_1$ and $\cU_2$ both comprise the four singleton
sets of nearest neighbors, we get the threshold $1$ version of (D0)--(D2). 
In this case, we claim that 
$\gamma_2=0$. Indeed, 
if a connected set of $0$s and $1$s in the initial configuration 
contains at least one $1$, then every site 
is in state 
$1$ at some time $t\ge 0$. Therefore, 
this is true for the infinite connected cluster 
of $0$s and $1$s in the initial state, if $q$ is small enough. This cluster must have a $2$
on its outside boundary, and therefore every one of its sites must eventually 
also turn into a $2$. Therefore, the final density of $2$s approaches $1$
as $q\to 0$, regardless of $p$. 

Determining $\gamma_2$, or even deciding if it has one of the 
two extreme values, appears to be a challenging problem. For example, 
consider nearest neighbor 
threshold growth in three dimensions given by appropriately changed (D0)--(D2): $N_i(x,t)$ counts sites in state $i$ 
among six nearest neighbors of $x$ at time $t$, and the threshold is $2$
or $3$. For threshold $2$ (whose polluted version was studied in \cite{GravnerHolroyd}), we do not have a guess for $\gamma_2$.
For threshold $3$, 
\cite{Ding} implies that 
$\gamma_2\ge 3$ and it seems possible that equality holds. For the three-dimensional modified 
model with threshold $3$, the method from \cite{GravnerHolroydSivakoff} implies that $\gamma_2=\infty$, as a $0$ which has two nearest neighbors $2$s in, say, $x$-coordinate can never become a $1$. In the initial state, 
such permanent $0$s have density $q^2$, far above the 
density on the order $q^3$ which suffices to block the spread of $2$s. 

 \bigskip
 
\noindent 6. {\it Multicolor rules\/}. We can generalize
the standard rule to any number $\kappa\ge 2$ of nonzero 
colors, with oriented linear graph on $\kappa+1$
states as the update graph. That is,  $\xi_t\in\{0,1,\ldots,\kappa\}^{\bZ^2}$, $\xi_0$ is a product measure with 
$\P(\xi_0(x)=i)=p_i>0$, and for $t\ge 0$: 
\begin{itemize}[leftmargin=1.25cm]
\item[(MD1)] if $\xi_t(x)=i<\kappa$, then $\xi_{t+1}(x)=i+1$ if $N_{i+1}(x,t)\ge 2$ and $\xi_{t+1}(x)=i$ otherwise;
\item[(MD2)] if $\xi_t(x)=\kappa$, then $\xi_{t+1}(x)=\kappa$.
\end{itemize}
Within the theme of the present paper, we inquire about conditions that 
ensure most sites reach the terminal state $\kappa$. For example, 
do there exist constants $\gamma_i>0$ so that $p_{i}\ll p_{i-1}^{\gamma_{i}}$,  $i=2,\ldots,\kappa$, implies that 
$\P(\text{the origin is eventually in state $\kappa$})\to 1$
as $p_1\to 0$? It follows from Theorem~\ref{thm:2sloseinmodified} that 
if the standard variant in (MD1) is replaced by its modified variant for at least one $i\in[0,\kappa-2]$, 
then
$\P(\xi_\infty(0)\le i+1)\to 1$ as $\max(p_{i+1},\ldots, p_\kappa)\to 0$.





\section*{Acknowledgments} 
All three authors were partially supported by the Simons Foundation. JG was also partially supported by the Slovenian Research Agency research program P1-0285.

\bibliographystyle{alpha}
\bibliography{refs}

\end{document}